\documentclass[12pt,reqno]{amsart}
\usepackage{amsmath,amstext,amssymb,amscd,hyperref}
\usepackage{verbatim}
\usepackage{enumerate}
\usepackage{mathrsfs}
\usepackage[dvipsnames,usenames]{xcolor}

\usepackage{amsthm}

\newtheorem{theorem}{Theorem}[section]
\newtheorem{lemma}[theorem]{Lemma}
\newtheorem{corollary}[theorem]{Corollary}
\newtheorem{proposition}[theorem]{Proposition}
\newtheorem{assumption}[theorem]{Assumption}

\theoremstyle{definition}
\newtheorem{definition}[theorem]{Definition}

\theoremstyle{remark}
\newtheorem{remark}[theorem]{Remark}

\numberwithin{equation}{section}

\def\t{\tau}

\def\G{\Gamma}
\def\O{\Omega}

\newcommand{\reals}{\mathbb{R}}

\newcommand{\norm}[1]{\left\|#1\right\|}

\author[Y. Guo]{Yanqiu Guo}
\address{Department of Mathematics \& Statistics \\Florida International University\\
Miami, Florida 33199, USA} \email{yanguo@fiu.edu}

\author[M. A. Rammaha]{Mohammad A. Rammaha}
\address{Department of
Mathematics \\ University of Nebraska-Lincoln \\
Lincoln, NE  68588-0130, USA} \email{mrammaha1@unl.edu}

\author[S. Sakuntasathien]{Sawanya Sakuntasathien}
\address{Department of Mathematics \\ Faculty of Science \\ Silpakorn University  \\
Nakhonpathom, 73000, Thailand} \email{sakuntasathien\_s@silpakorn.edu}

\title[Energy decay of a viscoelastic wave equation]
{Energy decay of a viscoelastic wave equation with supercritical nonlinearities}
\date{January 7, 2018}

\begin{document}

\maketitle

\begin{abstract}
This paper presents a study of the asymptotic behavior of the solutions for the history value problem of a viscoelastic wave equation which features a fading memory term as well as a supercritical source term and a frictional damping term:
\begin{align*}
\begin{cases}
u_{tt}- k(0) \Delta u -  \int_0^{\infty} k'(s) \Delta u(t-s) ds +|u_t|^{m-1}u_t =|u|^{p-1}u,   \text{\;\;\;\;in\;\;} \O \times (0,T), \\
u(x,t)=u_0(x,t), \quad \text{ in } \O \times (-\infty,0],
\end{cases}
\end{align*}
where $\O$ is a bounded domain in $\mathbb R^3$ with a Dirichl\'et boundary condition and $u_0$ represents the history value. A suitable notion of a potential well is introduced for the system, and  global existence of solutions is justified provided that the history value $u_0$ is taken from a subset of the potential well. Also, uniform energy decay rate is obtained which depends on the relaxation kernel $-k'(s)$ as well as the growth rate of the damping term. This manuscript complements our previous work \cite{GRSTT,GRS} where Hadamard well-posedness and the singularity formulation have been studied for the system. It is worth stressing the special features of the model, namely, the source term here has a supercritical growth rate and the memory term accounts to the full past history that goes back to $-\infty$.
\end{abstract}

\section {Introduction}\label{S1}
\subsection{The model}

Viscoelasticity is the property of materials that exhibit both viscous and elastic characteristics when undergoing deformation. 
Skin tissue is viscoelastic, which can be seen by pinching the skin and it takes time to recover back to its original flat position.
Synthetic polymers, wood, as well as metals at high temperature, display significant viscoelastic effects. In the nineteenth century, Boltzmann discovered that the behavior of viscoelastic materials should be modeled through constitutive relations that involve long but fading memory. In this paper, we study a system that models wave propagation in a viscoelastic medium under the influence of a frictional damping and an energy-amplifying strong source.

In particular, the model under investigation can be represented by a hyperbolic partial differential equation: 
\begin{align} \label{1.1}
\begin{cases}
u_{tt}- k(0) \Delta u -  \int_0^{\infty} k'(s) \Delta u(t-s) ds + |u_t|^{m-1}u_t=|u|^{p-1}u,\text{\;in\;} \O \times (0,T)\\
u(x,t)=0, \quad \text{ on }  \Gamma \times (-\infty,T) \\
u(x,t)=u_0(x,t), \quad \text{ in } \O \times (-\infty,0],
\end{cases}
\end{align}
where the unknown quantity $u(x,t)$ is a real-valued function defined on $\Omega \times (-\infty,T)$, and $\O\subset \reals^3$ is an open, bounded, and connected domain with a smooth boundary $\Gamma$. Notice that the history value of $u$ is assigned for all non-positive time $t\in (-\infty,0]$. The term $-\int_0^{\infty} k'(s) \Delta u(t-s) ds$ is a linear memory term that accounts to the full past history that goes back to $-\infty$. The kernel  $-k'(s)$ is often called the relaxation kernel, which characterizes the rate at which the memory fades. The term $|u_t|^{m-1}u_t$, $m\geq 1$ stands for a frictional damping that strongly dissipates energy, whereas the term $|u|^{p-1}u$, $p\in [1,6)$,  represents a nonlinear source of \emph{supercritical} growth rate which models an external force that amplifies energy and drives the system to possible instability. It is evident that the linear memory term, the damping term, and the source term are the three major players in the model. Their interactions stimulate many interesting phenomena which deserve careful investigation.

The present work relies on the rigorously justified Hadamard well-posedness result for (\ref{1.1}) in our previous paper \cite{GRSTT} coauthored with Titi and Toundykov. Moreover, in another recent paper \cite{GRS}, we have provided a delicate study on the formation of singularities  for (\ref{1.1}) under the scenario of both positive and negative initial energy. In order to have a more complete analysis of system (\ref{1.1}), in this paper we study the asymptotic behavior of global solutions in a potential well and obtain a uniform decay rate of energy.

Model (\ref{1.1}) describes wave propagation in a viscoelastic medium. Such phenomena are common in nature. For instance, in \cite{PZT} a viscoelastic wave model is introduced to describe acoustic signal propagation in a shallow sea, and it is mentioned in \cite{PZT} that the low-frequency acoustic waves propagate not only in the seawater, but also in the near-surface layers of the ocean bottom, together forming a geo-acoustic waveguide, so that the acoustic field is substantially influenced by the shear elasticity of sediments and rocks composing the bottom, as well as the sound attenuation in these materials.
On the other hand, as a consequence of the widespread use of polymers and other modern materials which exhibit stress relaxation, the theory of viscoelasticity has provided important applications in materials science and engineering. We refer the reader to \cite{Coleman, CN, FM, RHN} for mathematical foundation as well as applications for viscoelasticity.

We would like to point out that the topic of this work is within the field of mathematical control theory, an area of application-oriented mathematics that deals with the basic principles underlying the analysis and design of control systems. To control an object means to influence its behavior so as to achieve a desired goal. In our system (\ref{1.1}), intrinsic viscoelastic damping and external frictional damping mechanism acting on the system are responsible for  dissipation of its energy. 
The purpose of this line of study is to find conditions on initial state and determine certain amount of dissipation that are needed in order to obtain the optimal decay rate of the energy. In other words, the goal is to discover an adequate choice of the controls that can drive the system from a given initial state to a final given state, in a given time.

Let us briefly give an overview of some related results in the literature regarding the stability and the large time behavior for solutions to viscoelastic wave equations. The seminal papers \cite{Daf2,Daf1} by Dafermos were among earliest results on the asymptotic behavior of solutions to the equations of linear viscoelasticity at large time. Pata and Zucchi \cite{Pata} established the theory of finite-dimensional attractors for a damped hyperbolic equation with a linear memory. In an innovative work \cite{Cav5}, Cavalcanti and Oquendo studied the energy decay rate for a partially viscoelastic nonlinear wave equation subject to a nonlinear and localized frictional damping. Also, Berrimi and Messaoudi \cite{BM} examined the decay of solutions of a viscoelastic equation with a subcritical nonlinear source. We refer the reader to \cite{ACC,ACS,CCLW,DPZ,Fab1,LMM,LW} and the references therein for other related results.

The study of the interaction of nonlinear damping and source terms in wave equations was initiated by Georgiev and Todorova  \cite{G-T}. In this line of research, an important breakthrough was made by Bociu and Lasiecka in a series of papers \cite{BL3,BL2,BL1} where they provided a complete study of a wave equation with damping and supercritical sources in the interior and on the boundary of the domain. Indeed, a source term $|u|^{p-1} u$ is called subcritical if $1\leq p<3$, critical if $p=3$, and supercritical if $p>3$, in three space dimensions.  In the latter case  $p>3$, the source $u\mapsto |u|^{p-1} u$ is not locally Lipschitz from $H^1_0(\Omega)$ to $L^2(\Omega)$, which makes analysis of such problems more subtle. We mention here the references \cite{BGRT,BRT,GR,GR1,GR2,GRSTT,GRS,KR,PRT2,RW} that address hyperbolic equations influenced by supercritical sources.

In this paper, we follow the framework established by Lasiecka and Tataru \cite{LT} to study the asymptotic behavior of solutions to (\ref{1.1}). 
We also adopt important ideas from \cite{Cav5, LMM}. The novelty of our results can be seen from the following aspects.
\begin{itemize}
\item The source term in our model (\ref{1.1}) is supercritical, and the system is influenced by the full past history back to $-\infty$, unlike most related results in the literature which often dealt with subcritical sources and finite-time memories. Notice that, a viscoelastic model with a finite-time memory is typically given as an initial value problem, which is very different from our history value problem (\ref{1.1}).

\item We introduce a potential well in $L^2_{\mu}(\mathbb R^+,H^1_0(\Omega))$ which is particularly tailored for the history value problem (\ref{1.1}). 

\item Our proof for the energy decay rates is concise in the sense that we have successfully removed the compactness--uniqueness argument that was a common technique often used in the literature for the purpose of absorption of the lower-order terms (e.g., \cite{ACCRT,BRT,GR2,PP,PRT,RTW}).

\item In a few relevant papers (e.g., \cite{GR2,PP,PRT,RTW}), the energy decay estimate was conducted for data in a certain closed subset of the ``good" part of a potential well. Such restriction has been loosen in this work, so that the data admit a broader selection.  

\end{itemize}

\subsection{Summary of our previous results in \cite{GRSTT,GRS}}

We shall summarize the well-posedness and  blow-up results that we obtained for system (\ref{1.1}) in our papers \cite{GRSTT,GRS}.

Throughout, we denote
\begin{align*}
\mathbb R^+=[0,\infty),   \;\;\;\;\mathbb R&^-=(-\infty,0].
\end{align*}
Also, $C$ denotes a generic positive constant that may differ from line to line.

For the purpose of defining a proper function space for the history value $u_0$, we set $\mu(s)=-k'(s)$ and assume $k$ is monotone decreasing in $\mathbb R^+$. Thus, $\mu: \reals^+\longrightarrow \reals^+$, and in Assumption \ref{ass} precise assumptions on $\mu$ and $k$ will be imposed. We assume that the history value is represented by a function $u_0(x,t)$ defined for non-positive  times $t\in \mathbb R^-$.  In particular,  $u_0(x,t): \Omega\times \mathbb R^- \rightarrow \mathbb R$ belongs to a weighted Hilbert space $L^2_{\mu}(\mathbb R^-, H^1_0(\O))$, i.e.,
$$\|u_0\|^2_{L^2_{\mu}(\mathbb R^-, H^1_0(\O))}=\int_0^{\infty} \int_{\O} |\nabla u_0(x,-t)|^2 dx \mu(t) dt<\infty.$$
Also, we assume that $\partial_t u_0\in L^2_{\mu}(\reals^-,L^2(\O))$. Throughout, the standard $L^s(\O)$-norm will be denoted by:
\begin{align*}
\norm{u}_s=\norm{u}_{L^s(\O)}.
\end{align*}

The following assumptions will be imposed throughout the manuscript.
\begin{assumption} \label{ass}\leavevmode
\begin{itemize}
\item ~~ $m\geq 1$, $1\leq p<6$, $p\frac{m+1}{m}<6$;
\item ~~ $k\in C^2(\reals^+)$ such that $k'(s)<0$ for all $s>0$ and $k(\infty)=1$;
\item ~~ $\mu(s)=-k'(s)$ such that $\mu \in C^1(\reals^+)\cap L^1(\reals^+)$ and $\mu(s)> 0$, $\mu'(s)\leq 0$ for all $s>0$, and $\mu(\infty)=0$;
\item ~~ $u_0(x,t)\in L^2_{\mu}(\reals^-,H^1_0(\O))$,\, $\partial_t u_0(x,t)\in L^2_{\mu}(\reals^-,L^2(\O))$ and  such that\\ $u_0: \mathbb R^- \rightarrow H^1_0(\O)$ and
$\partial_t u_0(x,t): \mathbb R^- \rightarrow L^2(\O)$ are weakly continuous at $t=0$. In addition, for all $t\leq 0$, $u_0(x,t)=0$ on $\G$.
\end{itemize}
\end{assumption}

We define weak solutions of (\ref{1.1}) as follows.
\begin{definition} \label{def-weak}
A function $u(x,t)$ is said to be a \emph{weak solution} of (\ref{1.1}) defined on the time interval $(-\infty,T]$
provided $u\in C([0,T];H^1_0(\O))$ such that $u_t\in C([0,T];L^2(\O))\cap L^{m+1}(\O \times (0,T))$ with:
\begin{itemize}
\item $u(x,t)=u_0(x,t)\in L^2_{\mu}(\reals^-,H^1_0(\O))$ for $t\leq 0$;
\item The following variational identity holds for all $t\in [0,T]$  and all test functions $\phi \in \mathscr{F}$:
\begin{align} \label{weak}
&\int_{\O}u_t(t)\phi(t) dx -  \int_{\O}u_t(0)\phi(0) dx   -\int_0^t \int_{\O}u_t(\t)\phi_t(\t)dx d\t\notag\\
&+k(0)\int_0^t \int_{\O} \nabla u(\t) \cdot \nabla \phi(\t) dx d\t
+\int_0^t \int_0^{\infty} \int_{\O} \nabla u(\t-s) \cdot \nabla \phi(\t) dx k'(s) ds d\t\notag\\
&+\int_0^t \int_{\O} |u_t(\tau)|^{m-1}u_t(\tau)\phi(\t) dx d\t
=\int_0^t \int_{\O} |u(\tau)|^{p-1}u(\tau) \phi(\t) dx d\t,
\end{align}
where $$ \mathscr{F} = \Big\{ \phi: \,\, \phi \in C([0,T];H^1_0(\O))\cap L^{m+1}(\O \times (0,T)) \text{  with  } \phi_t\in C([0,T];L^2(\O)) \Big\}.$$
\end{itemize}
\end{definition}

For a solution $u(x,t)$, let us define the function 
\begin{align*}
w(x,t,s)=u(x,t)-u(x,t-s),\;\;\; s\geq 0.
\end{align*}
Also, we define the quadratic energy $\mathscr E(t)$ by:
\begin{align}\label{1.6}
\mathscr E(t)=\frac{1}{2}\left(\norm{u_t(t)}_2^2+\norm{\nabla u(t)}_2^2+\int_0^{\infty}\norm{\nabla w(t,s)}_2^2 \mu(s)ds \right).
\end{align}
\begin{remark}    \label{remk1}
In the definition (\ref{1.6}) of the quadratic energy $\mathscr E(t)$, the term $\frac{1}{2} \|u_t(t)\|_2^2$ represents the kinetic energy, whereas the term
$\frac{1}{2} \|\nabla u(t)\|_2^2$ is a potential energy. In addition, the term $\frac{1}{2} \int_0^{\infty} \|\nabla w(t,s)\|_2^2 \mu(s) ds$ can be also regarded as a potential energy due to the elastic property of the viscoelastic medium.   
\end{remark}

The following theorem from our paper \cite{GRSTT} is on the local existence and uniqueness of weak solutions of (\ref{1.1}).
\begin{theorem}[{\bf Short-time existence and uniqueness} \cite{GRSTT}] \label{thm-exist}
Assume the validity of Assumption \ref{ass}. Then, there exists a local (in time) weak solution $u$ to (\ref{1.1}) defined on the time interval $(-\infty,T]$ for some $T>0$ depending on the initial quadratic energy $\mathscr E(0)$. Furthermore, the following energy identity holds:
\begin{align} \label{EI-0}
&\mathscr E(t)+\int_0^t \int_{\O} |u_t|^{m+1} dx d\t-\frac{1}{2}\int_0^t \int_0^{\infty} \norm{\nabla w(\t, s)}_2^2 \mu'(s)ds d\t \notag\\
&=\mathscr E(0)+\int_0^t \int_{\O} |u|^{p-1}u u_t dx d\t.
\end{align}
If in addition we assume $u_0(0)\in L^{\frac{3(p-1)}{2}}(\O)$, then weak solutions of (\ref{1.1}) are unique.
\end{theorem}

\vspace{0.02 in}

\begin{remark}
If $u(t)\in L^{p+1}(\Omega)$ for all $t$ in the lifespan of the solution, we can define the \emph{total energy} $E(t)$ of system (\ref{1.1}) by
\begin{align}  \label{to-ene}
E(t)&=\mathscr E(t) -  \frac{1}{p+1}\|u(t)\|_{p+1}^{p+1} \notag \\
&=\frac{1}{2}\left(\norm{u_t(t)}_2^2+\norm{\nabla u(t)}_2^2+\int_0^{\infty}\norm{\nabla w(t,s)}_2^2 \mu(s)ds \right)- \frac{1}{p+1}\|u(t)\|_{p+1}^{p+1}.
\end{align}
It is readily seen that, in terms of the total energy $E(t)$, the energy identity (\ref{EI-0}) can be written as
\begin{align} \label{EI-1}
E(t)+\int_0^t \int_{\O} |u_t|^{m+1} dx d\t-\frac{1}{2}\int_0^t \int_0^{\infty} \norm{\nabla w(\t, s)}_2^2 \mu'(s)ds d\t =E(0).
\end{align}
In the energy identity (\ref{EI-1}), the term $\int_0^t \int_{\O} |u_t|^{m+1} dx d\t$ represents the dissipation due to the frictional damping, whereas the term $-\frac{1}{2}\int_0^t \int_0^{\infty} \norm{\nabla w(\t, s)}_2^2 \mu'(s)ds d\t$ stands for the dissipation that is originated from the viscous feature of the viscoelastic medium. Also, the total energy $E(t)$ is monotone decreasing in time, since differentiating (\ref{EI-1}) yields
\begin{align}   \label{e-decrease}
E'(t) = -\|u_t(t)\|_{m+1}^{m+1} + \frac{1}{2} \int_0^{\infty} \|\nabla w(t,s)\|_2^2 \mu'(s) ds\leq 0,
\end{align}
since we have assumed $\mu'(s)\leq 0$ for all $s\geq 0$.
\end{remark}

The next result states that weak solutions of (\ref{1.1}) depend continuously on the history.
\begin{theorem}[{\bf Continuous dependence on the history} \cite{GRSTT}]  \label{thm-cont}
In addition to Assumption \ref{ass}, assume that  $u_0(0)\in L^{\frac{3(p-1)}{2}}(\O)$.
If $u_0^n\in L^2_{\mu}(\reals^-, H^1_0(\O))$ is a sequence of history values such that $u_0^n\longrightarrow u_0$ in $L^2_{\mu}(\reals^-,H^1_0(\O))$ with
$u^n_0(0)\longrightarrow u_0(0)$ in $H^1_0(\O)$ and in $L^{\frac{3(p-1)}{2}}(\O)$, $\frac{d}{dt}u^n_0(0)\longrightarrow \frac{d}{dt}u_0(0)$ in $L^2(\O)$, then the corresponding weak solutions $u_n$ and $u$ of (\ref{1.1}) satisfy
\begin{align*}
u_n\longrightarrow u \text{\;\;in\;\;}  C([0,T];H^1_0(\O)) \text{\;\;\;and\;\;\;}
(u_n)_t\longrightarrow u_t  \text{\;\;in\;\;}  C([0,T];L^2(\O)).
\end{align*}
\end{theorem}

Furthermore, the following result states: if the damping dominates the source term, then the solution is global.
\begin{theorem}[{\bf Global existence} \cite{GRSTT}] \label{thm-global}
In addition to Assumption \ref{ass}, further assume $u_0(0) \in L^{p+1}(\O)$. If $m\geq p$, then the weak solution of (\ref{1.1}) is global.
\end{theorem}

However, if the initial energy $E(0)$ is negative and the source dominates the dissipation in the system, then the weak solution of (\ref{1.1}) blows up in finite time. Specifically, we have the following result.
\begin{theorem}[{\bf Blow-up of solutions with negative initial energy} \cite{GRS}]   \label{thm-blow}
Assume the validity of Assumption \ref{ass} and $E(0)<0$. If $p>\max\{m,\sqrt{k(0)}\}$, then the weak solution $u$ of (\ref{1.1}) blows up in finite time. More precisely, $\|\nabla u(t)\|_2 \rightarrow \infty$ as $t\rightarrow T^-_{max}$, for some $T_{\max}\in (0,\infty)$.
\end{theorem}

The next blow-up result deals with the scenario that the initial total energy is positive. 
We set the constant $d$ to be the mountain pass level: $$d:=\inf_{u\in H^1_0(\O)\backslash \{0\}} \sup_{ \lambda \geq 0} J(\lambda u), \text{ where  }  J(u):=\frac{1}{2}\norm{\nabla u}_2^2 - \frac{1}{p+1}\norm{u}_{p+1}^{p+1}.$$
It is well known that an equivalent definition for the constant $d$ is given by $d=\inf_{u\in \mathcal N} J(u)$, where  $\mathcal N$ is  the Nehari manifold defined by
\begin{align}    \label{Nehari-1}
\mathcal N=\{u\in H^1_0(\Omega)\backslash \{0\}:  \|\nabla u\|_2^2 = \|u\|_{p+1}^{p+1}\}.
\end{align}
Also, it can be shown, similar to the proof of Lemma \ref{lemm-d0}, that $d=\frac{p-1}{2(p+1)}\gamma^{-\frac{2(p+1)}{p-1}}$ where $\gamma$ is defined in
(\ref{gamma}).
In addition, we define the constants $y_0$ and $M$ by
\begin{align}   \label{con-1}
y_0:= \frac{p+1}{p-1}d,  \;\;\;   M:= \left(\frac{\sqrt{k(0)}+1}{2}\right)^{\frac{2}{p-1}} \frac{p-\sqrt{k(0)}}{p-1} d,\end{align}
provided $p>\sqrt{k(0)}>1$. It can be shown that $M<d$. We refer the reader to \cite{GRS} for  the details  of the proofs of these facts.

\begin{theorem}[{\bf Blow-up of solutions with positive initial energy} \cite{GRS}]   \label{thm-blow2}
 In addition to the validity of Assumption \ref{ass}, we assume that $0\leq E(0)<M$ and $\mathscr E(0) > y_0$, where the positive numbers $M$ and $y_0$ are defined in (\ref{con-1}). If $p>\max\{m,\sqrt{k(0)}\}$, then the weak solution $u$ of (\ref{1.1}) blows up in finite time. More precisely, $\|\nabla u(t)\|_2 \rightarrow \infty$ as $t\rightarrow T^-_{max}$, for some $T_{\max}\in (0,\infty)$.
\end{theorem}

\vspace{0.05 in}

\section{Main results}    \label{sec-main}
This section contains the statements of the main results:  global existence of weak solutions in a potential well, and  uniform decay  rates of energy.

\subsection{Global existence}
In order to state the global existence result, we shall introduce a suitable notion of a potential well for the history value problem (\ref{1.1}).

First, notice that the total energy $E(t)$ defined in (\ref{to-ene}) can be decomposed into the sum of the kinetic energy $\frac{1}{2} \|u_t(t)\|_2^2$ and the potential energy $I(t)$:
\begin{align}    \label{potential}
I(t):=\frac{1}{2} \left(\|\nabla u(t)\|_2^2+\int_0^{\infty} \|\nabla w(t,s)\|_2^2 \mu(s) ds   \right)
-\frac{1}{p+1} \|u(t)\|_{p+1}^{p+1}.
\end{align}

We define the set $\mathcal M$ in the space $L^2_{\mu} (\mathbb R^-, H^1_0(\Omega))$: 
\begin{align} \label{Nehari}
\mathcal M = \Big\{ v &\in L^2_{\mu}(\mathbb R^-,H^1_0(\Omega))\backslash \{0\}: \;v \text{\;is weakly continuous at\;} t=0,    \text{\;and\;\;} \notag\\
&\|\nabla v(0)\|_2^2  + \int_0^{\infty} \|\nabla v(0)-\nabla v(-s)\|_2^2 \mu(s) ds = \|v(0)\|_{p+1}^{p+1} \Big\} .
\end{align}
The definition of the set $\mathcal M$ is motivated by the concept of the Nehari manifold defined in (\ref{Nehari-1}).

For each $p\in [1,5]$, let $\gamma>0$ denote the best constant for the Sobolev inequality 
\begin{align}   \label{Sob}
\|u\|_{p+1} \leq \gamma \|\nabla u\|_2, 
\end{align}
i.e., 
\begin{align}   \label{gamma}
\gamma=\sup \{\|u\|_{p+1}: u\in H^1_0(\Omega), \|\nabla u\|_2=1\}. 
\end{align}

For $v\in L^2_{\mu}(\mathbb R^-,H^1_0(\Omega))$ such that $v$ is weakly continuous at $t=0$, we define the functional 
\begin{align}   \label{Iv}
\mathcal I(v):= \frac{1}{2} \left(\|\nabla v(0)\|_2^2+\int_0^{\infty} \|\nabla v(0)-\nabla v(-s)\|_2^2 \mu(s) ds   \right)
-\frac{1}{p+1} \|v(0)\|_{p+1}^{p+1}.
\end{align}

We remark that the potential energy $I(t)$ defined in (\ref{potential}) and the functional $\mathcal I(v)$ defined in (\ref{Iv}) are closely connected. Indeed, 
if $u_0\in L^2_{\mu}(\mathbb R^-, H^1_0(\Omega))$ is a history datum of our problem (\ref{1.1}), then clearly 
\begin{align}   \label{Iu0}
\mathcal I(u_0)=I(0).
\end{align}

\begin{lemma}    \label{lemm-d0}
Assume $1<p\leq 5$. Then 
\begin{align*}
d:=\inf_{v\in \mathcal M} \mathcal I(v) =\frac{p-1}{2(p+1)}\gamma^{-\frac{2(p+1)}{p-1}}>0.
\end{align*}
\end{lemma}
The proof for Lemma \ref{lemm-d0} will be given in Section \ref{pge}.

Next, we define the set $\mathcal W$ in $L^2_{\mu} (\mathbb R^-, H^1_0(\Omega))$ as
\begin{align*} 
\mathcal W=\{v\in L^2_{\mu}(\mathbb R^-,H^1_0(\Omega)): v \text{\;is weakly continuous at\;} t=0, \;
\mathcal I(v) < d  \}.
\end{align*}
We call the set $\mathcal W$ a \emph{potential well} in the space $L_{\mu}^2(\mathbb R^-;H^1_0(\Omega))$. Notice that, if $u_0\in \mathcal W$, then by (\ref{Iu0}), we see that $I(0)=\mathcal I(u_0)<d$, i.e., the initial potential energy is less than $d$.

Also, we define two disjoint subsets $\mathcal W_1$ and $\mathcal W_2$ of the potential well $\mathcal W$. In particular,
\begin{align*} 
&\mathcal W_1:=\left\{ v\in \mathcal W:  \|\nabla v(0)\|_2^2  + \int_0^{\infty} \|\nabla v(0)-\nabla v(-s)\|_2^2 \mu(s) ds > \|v(0)\|_{p+1}^{p+1}     \right \} \cup \{0\},      \\
&\mathcal W_2:=\left\{ v\in \mathcal W:  \|\nabla v(0)\|_2^2  + \int_0^{\infty} \|\nabla v(0) -\nabla v(-s)  \|_2^2 \mu(s) ds < \|v(0)\|_{p+1}^{p+1}  \right \}.  
  \end{align*}
It is clear that $\mathcal W_1 \cap \mathcal W_2 = \emptyset $ and $\mathcal W_1 \cup \mathcal W_2 = \mathcal W$.
\vspace{0.1 in}

We shall show that if $u_0\in \mathcal W_1$ and $E(0)<d$, then the solution is global (see Theorem \ref{thm-global1}). On the other hand, if $u_0\in \mathcal W_2$, then the solution might blow up in finite time. 
Indeed, we have the following corollary of Theorem \ref{thm-blow2}. 
\begin{corollary}
In addition to Assumption \ref{ass}, we assume $p> \max\{m,\sqrt{k(0)}\}$, $0\leq E(0)<M<d$ where $M$ is defined in (\ref{con-1}). If $u_0\in \mathcal W_2$, then the solution of (\ref{1.1}) blows up in finite time.
\end{corollary}

\begin{proof}
Since $u_0\in \mathcal W_2$, 
\begin{align*}
\|\nabla u_0(0)\|_2^2  + \int_0^{\infty} \|\nabla u_0(0) -\nabla u_0(-s)  \|_2^2 \mu(s) ds < \|u_0(0)\|_{p+1}^{p+1} \leq \gamma^{p+1} \|\nabla u_0(0)\|_2^{p+1},
\end{align*} 
where we have used (\ref{Sob}). Dividing both sides of the above inequality by $\|\nabla u_0(0)\|_2^2$ yields $1 < \gamma^{p+1}\|\nabla u_0(0)\|_2^{p-1}$, which implies that
$\|\nabla u_0(0)\|_2 > \gamma^{-\frac{p+1}{p-1}}$, and thus $\mathscr E(0) \geq \frac{1}{2}\|\nabla u_0(0)\|_2^2 > \frac{1}{2}\gamma^{-\frac{2(p+1)}{p-1}}=y_0$ where $y_0$ is defined in (\ref{con-1}). Then by Theorem \ref{thm-blow2}, the solution blows up in finite time.
\end{proof}

In order to state a global existence result for $u_0\in \mathcal W_1$, we define a translation of functions as follows. Let $u(t)$ be a solution for system (\ref{1.1}) on $(-\infty,T)$ with the history value $u_0\in L^2_{\mu}(\mathbb R^-, H^1_0(\Omega))$. For each $t\in (0,T)$, 
we define $u^t: \mathbb R^-\rightarrow H^1_0(\Omega)$ by 
\begin{align}   \label{def-ut}
u^t(\tau):=u(t+\tau),   \;\; \tau\in \mathbb R^-.
\end{align}
Note that, if we fix a $t\in (0,T)$ and restrict $u$ on $(-\infty,t]$, then by (\ref{def-ut}), the function $u^t$ defined on $\mathbb R^-$ is obtained from shifting $u$ to the left $t$ units. Therefore, $u^t$ defined on $\mathbb R^-$ is a translated version of $u$ restricted to $(-\infty,t]$.

In addition, by (\ref{Iv}) and (\ref{def-ut}), we have
\begin{align}   \label{pot}
\mathcal I(u^t)&=\frac{1}{2} \left(\|\nabla u^t(0)\|_2^2+\int_0^{\infty} \|\nabla u^t(0)-\nabla u^t(-s)\|_2^2 \mu(s) ds   \right)
-\frac{1}{p+1} \|u^t(0)\|_{p+1}^{p+1}   \notag\\
&=\frac{1}{2} \left(\|\nabla u(t)\|_2^2+\int_0^{\infty} \|\nabla u(t)-\nabla u(t-s)\|_2^2 \mu(s) ds   \right)
-\frac{1}{p+1} \|u(t)\|_{p+1}^{p+1}.
\end{align}
According to (\ref{potential}) and (\ref{pot}), we obtain
\begin{align}   \label{pote}
\mathcal I(u^t)=I(t), \text{\;\;for any\;} t\in (0,T),
\end{align}
where $u$ is a solution of our system (\ref{1.1}).

\begin{theorem}   \label{thm-global1}
Let $1<p\leq 5$ and assume the validity of Assumption \ref{ass}. Suppose $E(0) < d$. Let $u$ be a weak solution of (\ref{1.1}) on its maximal interval of existence $(-\infty,T)$. For each fixed $t\in (0,T)$, define $u^t: \mathbb R^-\rightarrow H^1_0(\Omega)$ by (\ref{def-ut}).
If $u_0\in \mathcal W_1$, then $u^t \in \mathcal W_1$ for all $t\in (0,T)$. Moreover, the solution $u$ is global, i.e., $T=\infty$. Furthermore,
\begin{align}    \label{EE}
0\leq E(t)\leq E(0)<d   \text{\;\;\;\;and\;\;\;\;}  0\leq \frac{p-1}{p+1} \mathscr E(t) \leq E(t) \leq \mathscr E(t),  \text{\;\;for all\;\;}  t\geq 0.
\end{align}
\end{theorem}

\begin{remark}
Recall the definition of the Nehari manifold $\mathcal N=\{u\in H^1_0(\Omega): \|\nabla u\|_2^2 = \|u\|_{p+1}^{p+1}\}$. 
It can be proven, similar to the proof of Lemma \ref{lemm-d0}, that 
$d=\inf_{u\in \mathcal N} J(u)=\frac{p-1}{2(p+1)}\gamma^{-\frac{2(p+1)}{p-1}}$ where $J(u)=\frac{1}{2}\|\nabla u\|_2^2 - \frac{1}{p+1} \|u\|_{p+1}^{p+1}$. Also, we can define a different potential well $\mathcal V=\{u\in H^1_0(\Omega): J(u)<d\}$ and the corresponding subset $\mathcal V_1=\{u\in \mathcal V:     \|\nabla u\|_2^2 > \|u\|_{p+1}^{p+1}\} \cup \{0\}$.
Under this setting, we are able to justify that if $E(0)<d$ and $u_0(0)\in \mathcal V_1$, then the solution $u$ is global and $u(t)\in \mathcal V_1$ for all $t\geq 0$. In fact, this result is easier to prove than Theorem \ref{thm-global1}. On the other hand, the assumption that $E(0)<d$ and $u_0(0)\in \mathcal V_1$ implies that $u_0\in \mathcal W_1$. Therefore, the global existence theorem stated in Theorem \ref{thm-global1} is a better result and more suitable for our system, in the sense that, its assumption allows  a broader selection of history data $u_0$, and takes fully advantage of the entire history.
\end{remark}

\vspace{0.02 in}

\subsection{Uniform energy decay rates}
The following result provides the uniform energy decay rate for global solutions in a potential well. The energy decay rate depends on the behavior of the damping term near the origin, as well as the decreasing rate of the relaxation kernel $\mu(s)$. In particular, if the damping is linear near the origin and the relaxation kernel decays to zero exponentially, then the energy also decays to zero exponentially fast. Otherwise, the energy decays polynomially. More precisely, we have the following theorem.
 
\begin{theorem}   \label{thm-decay}
In addition to Assumption \ref{ass}, we assume $1\leq m\leq 5$ and $1<p<5$. Suppose $E(0)< d$ and $u_0\in \mathcal W_1$. Then the total energy $E(t)$ decays to zero, i.e., $\lim_{t\rightarrow \infty}E(t)=0$, with the following exponential or polynomial decay rate, which depends on the damping growth rate $m$ and on the decay rate of the relaxation kernel $\mu(s)>0$. (In the following, $C$ is some positive constant.)

\emph{Case 1:} If $m=1$ and $\mu'(s)+C\mu(s) \leq 0$ for $s\geq 0$, then $E(t)\leq C(E(0)) e^{-\alpha t}$ for $t\geq 0$, where $\alpha$ depends on $E(0)$.

\emph{Case 2:} If $m>1$ and $\mu'(s)+C\mu(s) \leq 0$ for $s\geq 0$, then $E(t)\leq C(E(0)) (1+t)^{-\frac{2}{m-1}}$ for $t\geq 0$.

\emph{Case 3:} If $m=1$ and $\mu'(s)+C\mu(s)^r\leq 0$ for $s\geq 0$, where $r\in (1,2)$, and assume $M_0:=\sup_{t\in \mathbb R^-} \|\nabla u_0(t)\|_2<\infty$, then $E(t)\leq C(r,\sigma,M_0,E(0)) (1+t)^{-\frac{\sigma}{r-1}}$ for $t\geq 0$, for any $\sigma \in (0,2-r)$. If, in addition, we assume that $u_0$ is supported on the time interval $[-T_0,0]$ for some $T_0>0$, then $E(t)\leq C(r,M_0,T_0,E(0)) (1+t)^{-\frac{1}{r-1}}$ for $t\geq 0$.

\emph{Case 4:} If $m>1$ and $\mu'(s)+C\mu(s)^r \leq 0$ for $s\geq 0$, where $r\in (1,2)$, and assume $M_0:=\sup_{t\in \mathbb R^-} \|\nabla u_0(t)\|_2<\infty$, then $E(t)\leq C(r,\sigma,M_0,E(0)) (1+t)^{-\max\left\{\frac{\sigma}{r-1},\frac{2}{m-1}\right\}}$ for $t\geq 0$, for any $\sigma\in (0,2-r)$. If, in addition, we assume that $u_0$ is supported on the time interval $[-T_0,0]$ for some $T_0>0$, then $E(t)\leq C(r,M_0,T_0,E(0)) (1+t)^{-\max\left\{\frac{1}{r-1},\frac{2}{m-1}\right\}}$ for $t\geq 0$.
\end{theorem}

\begin{remark}
In Theorem \ref{thm-decay}, for the Cases 1 and 2, the assumption that $\mu'(s)+C\mu(s) \leq 0$ with $\mu(s)>0$ represents that $\mu(s)$ decays to zero exponentially fast, i.e., $0<\mu(s)\leq \mu(0) e^{-Cs}$. However, for the Cases 3 and 4, the assumption that $\mu'(s)+C\mu(s)^r \leq 0$, $r\in (1,2)$, with $\mu(s)>0$ is equivalent to say that $\mu(s)$ decays to zero polynomially fast, i.e., 
$0<\mu(s)\leq C(1+s)^{-\frac{1}{r-1}}$, where $\frac{1}{r-1}\in (1,\infty).$

\end{remark}

\begin{remark}
For the case $m>5$, the above energy decay results can also be obtained if $u\in L^{\infty}(\mathbb R^+, L^{\frac{3}{2}(m-1)}(\Omega))$. However, in order to have such regularity of solutions, the history value $u_0$ has to be smoother, and a new regularity result needs to be established,  which is beyond the scope of the manuscript.  
\end{remark}

\begin{remark}
For the sake of clarity and conciseness, we study the prototype $|u|^{p-1}u$ of sources, as well as the prototype $|u_t|^{m-1}u_t$ of frictional damping. Nevertheless, all the results in the paper can be readily generalized to accommodate a class of source terms in the form $f(u)$ with $|f'(s)|\leq C(|s|^{p-1}+1)$ and a class of damping terms in the form $g(u_t)$ such that $g$ is continuous and monotone increasing 
with $g(s)s>0$ for all $s>0$, as well as $a|s|^{m+1} \leq g(s)s \leq b|s|^{m+1}$, for all $|s|\geq 1$ and some positive constants $a$ and $b$. We stress that the exact energy decay rate actually depends on the behavior of $g$ near the origin. 
\end{remark}

\vspace{0.05 in}

\section{Proof of global existence of solutions}  \label{pge}
In this section, we prove the theorem of the global existence of solutions stated in Section \ref{sec-main}.
\subsection{Proof of Lemma \ref{lemm-d0}}
\begin{proof}
Let $v\in \mathcal M$. So by (\ref{Nehari}) and (\ref{Iv}), we see that
\begin{align}  \label{g-1}
\mathcal I(v)= \left(\frac{1}{2} - \frac{1}{p+1}    \right) \left(\|\nabla v(0)\|_2^2  + \int_0^{\infty} \|\nabla v(0)-\nabla v(-s)\|_2^2 \mu(s) ds \right).
\end{align}
Also, since $1<p\leq 5$, by (\ref{Nehari}) and (\ref{Sob}), we obtain 
\begin{align}  \label{g-2}
\|\nabla v(0)\|_2^2+   \int_0^{\infty} \|\nabla v(0) - \nabla v(-s)\|_2^2 \mu(s) ds = \|v(0)\|_{p+1}^{p+1} \leq \gamma^{p+1} \|\nabla v(0)\|_2^{p+1}.
\end{align}
Assume $v(0)=0$ in $H^1_0(\Omega)$, then by (\ref{g-2}), we see that $\int_0^{\infty} \|\nabla v(-s)\|_2^2 \mu(s) ds =0 $, i.e., $v=0$ in $L^2_{\mu}(\mathbb R^-, H^1_0(\Omega))$ contradicting $v\in \mathcal M$. Hence, we must have $v(0)\not=0$ in $H^1_0(\Omega)$. Dividing both sides of (\ref{g-2}) by $\|\nabla v(0)\|_2^2$ yields
$1\leq \gamma^{p+1} \|\nabla v(0)\|_2^{p-1}$, that is, $\|\nabla v(0)\|_2 \geq \gamma^{-\frac{p+1}{p-1}}$. Then by (\ref{g-1}), we obtain that 
$\mathcal I(v)\geq \frac{p-1}{2(p+1)} \gamma^{-\frac{2(p+1)}{p-1}}$ if $v\in \mathcal M$. It follows that 
\begin{align}   \label{geqq}
\inf_{v\in \mathcal M}\mathcal I(v)\geq \frac{p-1}{2(p+1)} \gamma^{-\frac{2(p+1)}{p-1}}.
\end{align}

Next, we show that $\inf_{v\in \mathcal M}\mathcal I(v) \leq \frac{p-1}{2(p+1)} \gamma^{-\frac{2(p+1)}{p-1}}$. To achieve this, we let $u\in H^1_0(\Omega)$ with $\|\nabla u\|_2=1$, and define $\tilde v \in L^2_{\mu}(\mathbb R^-,H^1_0(\Omega))$ as
\begin{align}   \label{ue-2}
\tilde v(t)=\|u\|_{p+1}^{-\frac{p+1}{p-1}}u, \text{\;\;for all\;\;} t\in \mathbb R^-.
\end{align}
We see that $\tilde v$ is invariant in time, and thus $\tilde v \in L^2_{\mu}(\mathbb R^-, H^1_0(\Omega))$. Indeed,
\begin{align*}
\|\tilde v\|_{L^2_{\mu}(\mathbb R^-, H^1_0(\Omega))}^2 = \int_0^{\infty}  \|\nabla \tilde v (-s)\|_2^2 \mu(s) ds
= \|u\|_{p+1}^{-\frac{p+1}{p-1}}  \|\nabla u\|_2^2 \int_0^{\infty} \mu(s) ds 
\end{align*}
where $\int_0^{\infty} \mu(s) ds  = \int_0^{\infty} - k'(s) ds  = k(0)-k(\infty) =  k(0)-1< \infty$.

Also we claim $\tilde v\in \mathcal M$. Indeed, by virtue of (\ref{ue-2}) and the fact that $\|\nabla u\|_2=1$, we calculate
\begin{align*}
&\|\nabla \tilde v(0)\|_2^2+\int_0^{\infty} \|\nabla \tilde v(0) 
-  \nabla  \tilde v(-s)  \|_2^2  \mu(s) ds \notag\\
&=\|u\|_{p+1}^{-\frac{2(p+1)}{p-1}} \|\nabla u\|_2^2 =  \|u\|_{p+1}^{-\frac{2(p+1)}{p-1}}
=\|\tilde v(0)\|_{p+1}^{p+1}.
\end{align*}
Consequently, we have
\begin{align}    \label{leqq}
\inf_{v\in \mathcal M}\mathcal I(v) \leq \mathcal I(\tilde v)
&= \left(\frac{1}{2}-\frac{1}{p+1}\right)  \| \tilde v(0)\|_{p+1}^{p+1}
=\frac{p-1}{2(p+1)} \|u\|_{p+1}^{-\frac{2(p+1)}{p-1}}  \notag\\
&\leq \frac{p-1}{2(p+1)}  \gamma^{-\frac{2(p+1)}{p-1}},
\end{align}
due to (\ref{gamma}). 

Combining (\ref{geqq}) and (\ref{leqq}), we conclude that $\inf_{v\in \mathcal M}\mathcal I(v)=\frac{p-1}{2(p+1)}  \gamma^{-\frac{2(p+1)}{p-1}}$.

\end{proof}

\subsection{Proof of Theorem \ref{thm-global1}}
\begin{proof}
Recall the total energy $E(t)$ is monotone decreasing due to (\ref{e-decrease}). Then, since $E(0)<d$ and by (\ref{pote}) we obtain
\begin{align}  \label{g-4}
\mathcal I(u^t)=I(t) \leq E(t)\leq E(0) < d, \;\;\text{for all}\;\;  t\in (0,T).
\end{align}
This shows that $u^t \in \mathcal W$ for all $t\in (0,T)$.  

Let $u_0\in \mathcal W_1$. We aim to prove that $u^t \in \mathcal W_1$ for all $t\in (0,T)$. If $u_0=0$, then by system (\ref{1.1}), the unique solution $u(t)$ remains zero for all time, and thus, $u^t\in \mathcal W_1$ for all $t\in (0,T)$. Next, we consider the case $u_0 \not=0$. Under such scenario, we argue by contradiction to show that $u^t\in \mathcal W_1$ for all $t\in (0,T)$. Assume there exists $t_1\in(0,T)$ such that $u^{t_1} \not \in \mathcal W_1$. Since $u^{t_1}\in \mathcal W$, and by virtue of the fact $\mathcal W_1 \cap \mathcal W_2 =\emptyset$ and $\mathcal W_1 \cup \mathcal W_2 = \mathcal W$,
it follows that $u^{t_1}\in \mathcal W_2$, that is, 
\begin{align}  \label{g01}
\|\nabla u(t_1)\|_2^2  + \int_0^{\infty} \|\nabla w(t_1,s)\|_2^2 \mu(s) ds < \|u(t_1)\|_{p+1}^{p+1}.
\end{align}
In addition, since $u_0\in \mathcal W_1$ and $u_0\not =0$, one has
\begin{align}    \label{g02}
\|\nabla u(0)\|_2^2  + \int_0^{\infty} \|\nabla w(0,s)\|_2^2 \mu(s) ds > \|u(0)\|_{p+1}^{p+1}.
\end{align}

By the definition of weak solutions, i.e., Definition \ref{def-weak}, we know that $u\in C([0,T);H^1_0(\Omega))$ and $u_t\in C([0,T);L^2(\Omega))$. Also, by the energy identity (\ref{EI-0}), we see that $\mathscr E(t)$ is continuous. Moreover, (\ref{1.6}) shows
\begin{align*}
\int_0^{\infty} \|\nabla w(t,s)\|_2^2 \mu(s) ds = 2\mathscr E(t) - \|u_t(t)\|_2^2 - \|\nabla u(t)\|_2^2.
\end{align*}
Therefore, $\int_0^{\infty} \|\nabla w(t,s)\|_2^2 \mu(s) ds$ is continuous in $t$.

In addition, by using the fact $u\in C([0,T);H^1_0(\Omega))$, it is readily to derive that $\|u(t)\|_{p+1}$ is also continuous in time. Indeed, for $p\in (1,5]$,
the function $F(y)=|y|^{p+1}$ is $C^1$ with its monotone increasing derivative $F'(y)=(p+1)|y|^{p-1}y$, and therefore, by the mean-value theorem, and the H\"older's inequality, we have, for any $t$, $\tau\in [0,T)$,
\begin{align}   \label{g-10}
\left|  \|u(t)\|_{p+1}^{p+1}  - \|u(\tau)\|_{p+1}^{p+1}   \right|
&=\left|\int_{\Omega} |u(x,t)|^{p+1}-|u(x,\tau)|^{p+1} dx \right|   \notag\\
&\leq C\int_{\Omega}   \left(|u(x,t)|^{p}+    |u(x,\tau)|^{p}\right)   |u(x,t)-u(x,\tau)| dx \notag\\
&\leq C     \left( \|u(t)\|_{\frac{6}{5}p}^{p} +  \|u(\tau)\|_{\frac{6}{5}p}^{p} \right)    \|u(t)-u(\tau)\|_6  \notag\\
&\leq C   \left( \|\nabla u(t)\|_2^{p} +  \|\nabla u(\tau)\|_2^{p} \right)    \|\nabla u(t) - \nabla u(\tau)\|_2,
\end{align}
where we have used $p\in (1,5]$ and the imbedding $H^1_0(\Omega)\hookrightarrow L^6(\Omega)$ in three dimensions. Since $\|\nabla u(t)\|_2$ is continuous on $[0,T)$, we obtain from (\ref{g-10}) that $\|u(t)\|_{p+1}$ is also continuous on $[0,T)$.

As a result, due to the continuity of $\|\nabla u(t)\|_2$, $\int_0^{\infty} \|\nabla w(t,s)\|_2^2 \mu(s) ds$, and $\|u(t)\|_{p+1}$, we may apply the intermediate value theorem to obtain from (\ref{g01}) and (\ref{g02}) that there exists at least one point $\tilde t \in (0,t_1)$ such that
\begin{align}  \label{g-3}
\|\nabla u(\tilde t)\|_2^2  + \int_0^{\infty} \|\nabla w(\tilde t,s)\|_2^2 \mu(s) ds = \|u(\tilde t)\|_{p+1}^{p+1}.
\end{align}
Let $t^*$ be the supremum of all $\tilde t\in (0,t_1)$ such that (\ref{g-3}) holds. Clearly, $t^* \in (0,t_1)$, and (\ref{g-3}) is valid for $\tilde t=t^*$, and $u^t \in \mathcal W_2$ for all $t\in (t^*,t_1]$. 

\emph{Case 1: $u^{t^*}\not=0$}. Then $u^{t^*} \in \mathcal M$. Thus, from Lemma \ref{lemm-d0} and identity (\ref{pote}), we know that 
$I(t^*)=\mathcal I(u^{t^*}) \geq d$, which contradicts (\ref{g-4}). 

\emph{Case 2: $u^{t^*}=0$}. Since $u^t \in \mathcal W_2$ for all $t\in (t^*,t_1]$, we have 
\begin{align}   \label{g-8}
 \|\nabla u(t)\|_2^2  + \int_0^{\infty} \|\nabla w(t,s)\|_2^2 \mu(s) ds < \|u(t)\|_{p+1}^{p+1} \leq \gamma^{p+1} \|\nabla u(t)\|_2^{p+1}, 
 \end{align} 
for all $t\in (t^*,t_1]$, where we have used (\ref{Sob}). Clearly, from (\ref{g-8}), one has $\|\nabla u(t)\|_2>0$ for any $t\in (t^*,t_1]$. Thus, we may divide both sides of (\ref{g-8}) by $\|\nabla u(t)\|_2^2$ to obtain that 
$1< \gamma^{p+1} \|\nabla u(t)\|_2^{p-1}$, which implies $\|\nabla u(t)\|_2>\gamma^{-\frac{p+1}{p-1}}>0$ for all $t\in (t^*,t_1]$, and thus by the continuity, one has
$\|\nabla u(t^*)\|_2 \geq \gamma^{-\frac{p+1}{p-1}}>0$ which contradicts $u^{t^*}=0$.

This completes the proof for the claim that $u^t \in \mathcal W_1$ for all $t\in (0,T)$. Therefore, $\|u(t)\|_{p+1}^{p+1} < \|\nabla u(t)\|_2^2 + \int_0^{\infty} \|\nabla w(t,s)\|_2^2 \mu(s) ds   
\leq 2 \mathscr E(t)$ for all $t\in (0,T)$. It follows that
\begin{align*}  
E(t)=\mathscr E(t)-\frac{1}{p+1} \|u(t)\|_{p+1}^{p+1}
>  \mathscr E(t)-\frac{2}{p+1} \mathscr E(t) = \frac{p-1}{p+1} \mathscr E(t) \geq 0.
\end{align*}
The above inequality along with (\ref{g-4}) justifies (\ref{EE}).

It remains to show that $u(t)$ is global in time. By (\ref{g-4}), $I(t)<d$ for all $t\in (0,T)$, that is, 
\begin{align}     \label{g-5}
I(t)=\frac{1}{2} \left(\|\nabla u(t)\|_2^2+\int_0^{\infty} \|\nabla w(t,s)\|_2^2 \mu(s) ds   \right)
-\frac{1}{p+1} \|u(t)\|_{p+1}^{p+1} <d.
\end{align}
Also, we have shown that $u^t \in \mathcal W_1$ for each $t\in (0,T)$, that is,
\begin{align}  \label{g-9}
\|\nabla u(t)\|_2^2  + \int_0^{\infty} \|\nabla w(t,s)\|_2^2 \mu(s) ds > \|u(t)\|_{p+1}^{p+1}.
\end{align}

Combining (\ref{g-5}) and (\ref{g-9}) yields
\begin{align}   \label{g-6}
\left(\frac{1}{2}-\frac{1}{p+1}\right) \|u(t)\|_{p+1}^{p+1} <d, \text{\;\;for all\;\;} t\in (0,T).
\end{align}
Also, from the energy identity (\ref{EI-1}), we see that
\begin{align}  \label{g-7}
\mathscr E(t) \leq E(0)+ \frac{1}{p+1}\|u(t)\|_{p+1}^{p+1}, \text{\;\;for all\;\;} t\in (0,T).
\end{align}
From (\ref{g-6}) and (\ref{g-7}), we conclude that $\mathscr E(t)$ is uniformly bounded for all $t\in (0,T)$.
This implies that the solution is global and the maximal existence time $T=\infty$.
\end{proof}

\vspace{0.1 in}

\section{Proof for the uniform energy decay rates}
This section is devoted for deriving the uniform energy decay rate for solutions of (\ref{1.1}). In particular, we shall prove Theorem \ref{thm-decay}. The first step is to provide a stabilization estimate. 
\subsection{A stabilization estimate}   \label{sec-stab}
For sake of convenience, let $\textbf{D}(t)$ be given by:
\begin{align*}
\textbf{D}(t)=\int_0^t \|u_t(\tau)\|_{m+1}^{m+1} d\tau  - \frac{1}{2} \int_0^t  \int_0^{\infty} \|\nabla w(\tau,s)\|_2^2  \mu'(s) ds d\tau.
\end{align*}
Thus, the energy identity (\ref{EI-1}) can be written in the following compact form:
\begin{align}  \label{EI-1'}
E(t)+\textbf{D}(t)=E(0).
\end{align}
Here, $\textbf{D}(t)$ represents the dissipations from both of the frictional damping and the viscous effect due to the viscoelastic medium.

The following estimate shows how the dissipation $\textbf{D}(t)$ controls the total energy $E(t)$.

\begin{proposition}   \label{prop-pert}
In addition to Assumption \ref{ass}, suppose $1\leq m\leq 5$ and $1<p<5$. Also assume $E(0)<d$ and $u_0\in \mathcal W_1$.
We consider the following two cases:

\emph{Case A}: If $\mu'(s)+C\mu(s)\leq 0$ for all $s\geq 0$, then 
\begin{align*}    
E(t)\leq C(E(0)) \left(\textbf{D}(t)^{\frac{2}{m+1}}   +  \textbf{D}(t) \right),
\end{align*}
for sufficiently large $t$ depending on $E(0)$. 

\emph{Case B}: If $\mu'(s) + C \mu(s)^r \leq 0$ for all $s\geq 0$, where $r\in (1,2)$, and assume $M_0:=\sup_{\tau\in \mathbb R^-}\|\nabla u_0(\tau)\|_{2}<\infty$, then
\begin{align*}   
E(t)\leq C(r,\sigma,M_0,E(0)) \textbf{D}(t)^{\frac{\sigma}{\sigma+r-1}}  +  C(E(0)) \left( \textbf{D}(t)^{\frac{2}{m+1}} + \textbf{D}(t) \right),
\end{align*}
for sufficiently large $t$ depending on $E(0)$, and for any $\sigma \in (0,2-r)$.
\end{proposition}

\begin{proof}
First we point out that, by Theorem \ref{thm-global1}, $u$ is a global solution.
 
Next we show that $u\in L^{m+1}(\Omega \times (0,t))$ for any $t>0$. In fact, since both $u$ and $u_t$ in $C([0,t];L^2(\Omega))$, we may write
\begin{align*}
\int_0^t \int_{\Omega} |u(\tau)|^{m+1} dx d\tau 
&= \int_0^t \int_{\Omega} \left| \int_0^\tau u' ds + u(0) \right|^{m+1} dx d\tau \notag\\
&\leq C(t) \left(\|u_t\|^{m+1}_{L^{m+1}(\Omega \times (0,t))} + \|u(0)\|_{m+1}^{m+1}\right) < \infty
\end{align*}
due to the assumption $u_0(0)\in L^{m+1}(\Omega)$ and the regularity of solutions: $u_t\in L^{m+1}(\Omega\times (0,t))$. 

Consequently, in Definition \ref{def-weak} of weak solutions, we may set the test function $\phi=u$, since $u$ satisfies the required regularity. Then
\begin{align} \label{w1}
&\int_{\O}u_t(t)u(t) dx -  \int_{\O}u_t(0)u(0) dx   -\int_0^t \|u_t(\tau)\|_2^2 d\t\notag\\
&+k(0)\int_0^t \|\nabla u(\tau)\|_2^2 d\t +\int_0^t \int_0^{\infty} \int_{\O} \nabla u(\t-s) \cdot \nabla u(\t) dx k'(s) ds d\t\notag\\
&+\int_0^t \int_{\O} |u_t(\tau)|^{m-1}u_t(\tau)u(\t) dx d\t =\int_0^t \|u(\tau)\|_{p+1}^{p+1}  d\t.
\end{align}
Since $w(x,\tau,s)=u(x,\tau)-u(x,\tau-s)$, we write
\begin{align}   \label{w2}
&\int_0^t \int_0^{\infty} \int_{\O} \nabla u(\t-s) \cdot \nabla u(\t) dx k'(s) ds d\t \notag\\
&=\int_0^t \int_0^{\infty} \int_{\O} [\nabla u(\t-s)- \nabla u(\tau)] \cdot \nabla u(\t) dx k'(s) ds d\t
+\int_0^t \int_0^{\infty} \int_{\O} |\nabla u(\tau)|^2 dx  k'(s) ds d\tau \notag\\
& = -\int_0^t \int_0^{\infty} \int_{\O} \nabla w(\tau,s) \cdot \nabla u(\t) dx k'(s) ds d\t + (1-k(0)) \int_0^t  \|\nabla u(\tau)\|_2^2  d\tau,
\end{align}
where we used the assumption $k(\infty)=1$.

Substituting (\ref{w2}) into (\ref{w1}) yields
\begin{align}  \label{w3}
&\int_{\O}u_t(t)u(t) dx -  \int_{\O}u_t(0)u(0) dx   -\int_0^t \|u_t(\tau)\|_2^2 d\t \notag\\
&+\int_0^t \|\nabla u(\tau)\|_2^2 d\t -\int_0^t \int_0^{\infty} \int_{\O} \nabla w(\tau,s) \cdot \nabla u(\t) dx k'(s) ds d\t \notag\\
&+\int_0^t \int_{\O} |u_t(\tau)|^{m-1}u_t(\tau)u(\t) dx d\t =\int_0^t \|u(\tau)\|_{p+1}^{p+1}  d\t.
\end{align}
By using the definition (\ref{1.6}) of the quadratic energy $\mathscr E(t)$, we write (\ref{w3}) as
\begin{align}   \label{w4}
2\int_0^t \mathscr E(\tau) d\tau 
&= \int_0^t \norm{u_t(\tau)}_2^2 d\tau + \int_0^t \norm{\nabla u(\tau)}_2^2  d\tau
+ \int_0^t \int_0^{\infty}\norm{\nabla w(\tau,s)}_2^2 \mu(s) ds  d\tau \notag\\
& = -\int_{\O}u_t(t)u(t) dx +  \int_{\O}u_t(0)u(0) dx   + 2\int_0^t \|u_t(\tau)\|_2^2 d\t \notag\\
& \hspace{0.2 in} +\int_0^t \int_0^{\infty}\norm{\nabla w(\tau,s)}_2^2 \mu(s) ds  d\tau
 + \int_0^t \int_0^{\infty} \int_{\O} \nabla w(\tau,s) \cdot \nabla u(\t) dx k'(s) ds d\t   \notag\\
& \hspace{0.2 in} - \int_0^t \int_{\O} |u_t(\tau)|^{m-1}u_t(\tau)u(\t) dx d\t + \int_0^t \|u(\tau)\|_{p+1}^{p+1}  d\t.
\end{align}
Now we estimate each term on the right-hand side of (\ref{w4}).  \\

\vspace{0.1 in}

\textbf{1. Estimate for $\left|\int_0^t \int_0^{\infty} \int_{\O} \nabla w(\tau,s) \cdot \nabla u(\t) dx k'(s) ds d\t \right|$.}

Recall that $\mu(s)=-k'(s)>0$. By applying the Cauchy--Schwarz inequality as well as the Young's inequality, we have
\begin{align}  \label{w5p}
&\left|\int_0^t \int_0^{\infty} \int_{\O} \nabla w(\tau,s) \cdot \nabla u(\t) dx k'(s) ds d\t \right|  \notag\\
&\leq \epsilon \int_0^t \int_0^{\infty} \|\nabla u(\tau)\|_2^2 (-k'(s)) ds d\tau + C_{\epsilon}  \int_0^t \int_0^{\infty} \|\nabla w(\tau,s)\|_2^2 \mu(s) ds d\tau  \notag\\
&\leq 2\epsilon (k(0)-1) \int_0^t \mathscr E(\tau) d\tau  +  C_{\epsilon} \int_0^t \int_0^{\infty} \|\nabla w(\tau,s)\|_2^2   \mu(s) ds d\tau.
\end{align}

\vspace{0.1 in}

\textbf{2. Estimate for $\left|-\int_{\O}u_t(t)u(t) dx +  \int_{\O}u_t(0)u(0) dx \right|$.}

Let us recall the Poincar\'e inequality: $\|u\|_2 \leq C\|\nabla u\|_2$. Thus,
\begin{align*}
\left| \int_{\O}u_t(t)u(t) dx \right| \leq \frac{1}{2}\left(\|u_t(t)\|_2^2 + \|u(t)\|_2^2 \right)
\leq C\left(\|u_t(t)\|_2^2 + \|\nabla u(t)\|_2^2 \right)
\leq C \mathscr E(t).
\end{align*}
Also by (\ref{EE}), we have $\mathscr E(t)\leq \frac{p+1}{p-1}E(t)$ for all $t\geq 0$. It follows that
\begin{align}  \label{w5}
&\left|-\int_{\O}u_t(t)u(t) dx +  \int_{\O}u_t(0)u(0) dx \right| \notag\\
&\leq C(\mathscr E(t)+\mathscr E(0)) \leq C \left(\frac{p+1}{p-1}\right) \left[E(t) + E(0)\right] \leq C \left(\frac{p+1}{p-1}\right) \left[2E(t) + \textbf{D}(t) \right],\end{align}
where we have used the energy identity (\ref{EI-1'}).

\vspace{0.1 in}

\textbf{3. Estimate for $\int_0^t \|u_t(\tau)\|_2^2 d\tau$.}

By using the H\"older's inequality, we have
\begin{align}    \label{w15}
\int_0^t \int_{\O} |u_t(\tau)|^2 dx d\tau \leq \left(t |\Omega|\right)^{\frac{m-1}{m+1}} \left(\int_0^t \int_{\O} |u_t(\tau)|^{m+1} dx d\tau\right)^{\frac{2}{m+1}} \leq t |\Omega|  {\textbf D}(t)^{\frac{2}{m+1}},
\end{align}
for $t\geq 1/|\Omega|$. We remark that this estimate shows that how the dissipation ``controls" the kinetic energy. If $m=1$, then we have the linear estimate  $\int_0^t \|u_t(\tau)\| d\tau  \leq t\Omega  {\textbf D}(t)$; whereas, if the $m>1$, then we have the nonlinear estimate
$\int_0^t \|u_t(\tau)\| d\tau \leq t \Omega     {\textbf D}(t)^{\frac{2}{m+1}}$. Although estimate (\ref{w15}) is fairly easy to derive, it is a crucial element for the determination of the polynomial or the exponential decay rate of the total energy.

\vspace{0.1 in}

\textbf{4. Estimate for} $\int_0^t \|u(\tau)\|_{p+1}^{p+1} d\tau$, $1<p<5$.

The following estimate for the term $\int_0^t \|u(\tau)\|_{p+1}^{p+1} d\tau$ is essential for the whole proof.

From (\ref{EE}), we know that $\mathscr E(t) < \frac{p+1}{p-1} E(t)$. 
Also $E(\tau)\leq E(0)$, for all $\tau\geq 0$. Therefore,
\begin{align}   \label{w3-3}
\|\nabla u(\tau)\|_2^2 \leq 2 \mathscr E(\tau) < \frac{2(p+1)}{p-1} E(\tau) \leq   \frac{2(p+1)}{p-1} E(0), \text{\;\;for all\;\;} \tau\geq 0.
\end{align}

By (\ref{Sob}) and (\ref{w3-3}), we have
\begin{align}      \label{w3-1}
\|u(\tau)\|_{p+1}^{p+1} \leq \gamma^{p+1} \|\nabla u(\tau)\|_2^{p+1} &=    \gamma^{p+1}  \|\nabla u(\tau)\|_2^2    \|\nabla u(\tau)\|_2^{p-1}   \notag\\
&\leq \gamma^{p+1}   2 \mathscr E(\tau)  \left(\frac{2(p+1)}{p-1} E(0)\right)^{\frac{p-1}{2}}.
\end{align}
Since $E(0)< d=\frac{p-1}{2(p+1)}  \gamma^{-\frac{2(p+1)}{p-1}}$, 
one has $\gamma^{p+1}\left( \frac{2(p+1)}{p-1} E(0)\right)^{\frac{p-1}{2}}   <1$. 
As a result,
\begin{align}      \label{w3-2}
\int_0^t \|u(\tau)\|_{p+1}^{p+1}  d\tau \leq  \tilde C(E(0)) \int_0^t  \mathscr E(\tau) d\tau,
\end{align}
where 
\begin{align}    \label{tildeC}
\tilde C(E(0)) :=  2 \gamma^{p+1} \left(\frac{2(p+1)}{p-1} E(0)\right)^{\frac{p-1}{2}} <2.
\end{align}

Notice that the left-hand side of (\ref{w4}) is $2\int_0^t \mathscr E(\tau) d\tau$. Therefore, in (\ref{w3-2}) it is crucial that $\tilde C(E(0))$ is strictly less than 2.

\vspace{0.1 in}

\textbf{5. Estimate for $\left|\int_0^t \int_{\O} |u_t(\tau)|^{m-1}u_t(\tau)u(\tau) dx d\tau \right|$ .}

By using the H\"older's inequality, we obtain
\begin{align}    \label{w5-1}
\left|\int_0^t \int_{\O}  |u_t|^{m-1}u_t u  dx d\tau\right| 
&\leq  \left(\int_0^t \int_{\Omega} |u_t|^{m+1} dx d\tau\right)^{\frac{m}{m+1}}
    \left(\int_0^t \int_{\Omega} |u|^{m+1} dx d\tau\right)^{\frac{1}{m+1}}   \notag\\
&\leq \textbf{D}(t)^{\frac{m}{m+1}}    \|u\|_{L^{m+1}(\Omega\times (0,t))}.
    \end{align}

By the imbedding $H^1(\Omega) \hookrightarrow L^{m+1}(\Omega)$ in three dimensions for $1\leq m\leq 5$, together with Poincar\'e's inequality, as well as estimate (\ref{w3-3}), we deduce 
\begin{align}    \label{w5-2}
\|u\|_{L^{m+1}(\Omega\times (0,t))}^{m+1} \leq C\int_0^t \|\nabla u\|_2^{m+1} d\tau
&=  C\int_0^t \|\nabla u\|_2^{m-1}  \|\nabla u\|^2_2 d\tau  \notag\\
&\leq C\left(\frac{2(p+1)}{p-1}E(0)\right)^{\frac{m-1}{2}}   \int_0^t \mathscr E(\tau) d\tau.
\end{align}
Combining (\ref{w5-1}) and (\ref{w5-2}) yields
\begin{align}    \label{w21}
\left|\int_0^t \int_{\O}  |u_t|^{m-1}u_t u  dx d\tau\right| 
&\leq  C(E(0)) \textbf{D}(t)^{\frac{m}{m+1}} \left(\int_0^t \mathscr E(\tau)  d\tau \right)^{\frac{1}{m+1}}   \notag\\
&\leq \epsilon \int_0^t \mathscr E(\tau)  d\tau + C(\epsilon, E(0)) \textbf{D}(t),
\end{align}
where we have used the Young's inequality.

\vspace{0.1 in}

\textbf{6. Estimate for $\int_0^t \int_0^{\infty} \|\nabla w(\tau,s)\|_2^2 \mu(s) ds d\tau$.}

\emph{Case A:  $\mu'(s) + C\mu(s) \leq 0$}. Since $\mu(s)>0$ for all $s\geq 0$, then $\mu(s)$ decays to zero exponentially fast, that is, $0<\mu(s)\leq \mu(0) e^{-Cs}$. In this case, the estimate is simple.
\begin{align}    \label{w29}
\int_0^t \int_0^{\infty} \|\nabla w(\tau,s)\|_2^2 \mu(s) ds d\tau
\leq -C\int_0^t \int_0^{\infty} \|\nabla w(\tau,s)\|_2^2 \mu'(s) ds d\tau
\leq C\textbf{D}(t).
\end{align}
Notice that the term $\int_0^t \int_0^{\infty} \|\nabla w(\tau,s)\|_2^2 \mu(s) ds d\tau$ in (\ref{w29}) can be considered as the time integral of a potential energy due to the elasticity (see Remark \ref{remk1}), while the term $-\int_0^t \int_0^{\infty} \|\nabla w(\tau,s)\|_2^2 \mu'(s) ds d\tau$ represents the dissipation due to the viscosity (see (\ref{EI-1})). Therefore, inequality (\ref{w29}) can be interpreted as that the viscosity ``controls" the elasticity in the viscoelastic medium.

\vspace{0.1 in}

\emph{Case B: $\mu'(s) + C \mu(s)^r \leq 0$, where $1<r<2$.} 
Since $\mu(s)>0$, then $\mu(s)$ decays polynomially fast, that is, 
$0<\mu(s)\leq C(1+s)^{-\frac{1}{r-1}}$, where $\frac{1}{r-1}\in (1,\infty)$.
Inspired by an idea in \cite{Cav5}, we calculate
\begin{align}    \label{cB-1}
&\int_0^t \int_0^{\infty} \|\nabla w(\tau,s)\|_2^2 \mu(s) ds d\tau  
=\int_0^t \int_0^{\infty} \int_{\Omega} |\nabla w(\tau,s)|^2  \mu(s) dx ds  d\tau \notag\\
&=\int_0^t \int_0^{\infty} \int_{\Omega} \left(|\nabla w(\tau,s)|^{\frac{2(r-1)}{\sigma+r-1}} \mu(s)^{\frac{(1-\sigma)(r-1)}{\sigma+r-1}}\right)
\left(|\nabla w(\tau,s)|^{\frac{2\sigma}{\sigma+r-1}} \mu(s)^{ \frac{\sigma r}{\sigma+r -1}  } \right)dx ds d\tau  \notag\\
&\leq \left(\int_0^t \int_0^{\infty} \int_{\Omega} |\nabla w(\tau,s)|^2  \mu(s)^{1-\sigma} dx ds  d\tau\right)^{\frac{r-1}{\sigma+r-1}}     \notag\\
&\hspace{1 in}\left(\int_0^t \int_0^{\infty} \int_{\Omega} |\nabla w(\tau,s)|^2 \mu(s)^{r} dx ds  d\tau\right)^{\frac{\sigma}{\sigma+r-1}},
\end{align}
where we have used the H\"older's inequality,  $0<\sigma<1$, $1<r<2$.

Since $\mu(s)\leq C(1+s)^{-\frac{1}{r-1}}$, we have
\begin{align}  \label{cB-5}
\int_0^{\infty } \mu(s)^{1-\sigma}ds  \leq C \int_0^{\infty}  (1+s)^{-\frac{1-\sigma}{r-1}}    ds = C(r,\sigma)  <\infty, 
\end{align}
by assuming $\frac{1-\sigma}{r-1}>1$, i.e., $\sigma+r<2$.

By (\ref{w3-3}) and by assuming that $\|\nabla u_0(t)\|_{2}$ is uniformly bounded on $\mathbb R^-$, we see that $u(t)$ is uniformly bounded in $H^1_0(\Omega)$ for all $t\in \mathbb R$. 
By setting $M_0=\sup_{\tau\in \mathbb R^-} \|\nabla u(\tau)\|_2$ and using (\ref{w3-3}), we find that
\begin{align}    \label{cB-7}
\|\nabla u(\tau)\|_2^2 \leq \max \left\{  M_0^2,  \frac{2(p+1)}{p-1} E(0)     \right\}, \text{\;\;for all\;\;} \tau\in \mathbb R.
\end{align}
Recalling $w(\tau,s)=u(\tau)-u(\tau-s)$, and by (\ref{cB-7}), we have
\begin{align}   \label{cB-2}
&\int_0^t \int_0^{\infty} \int_{\Omega} |\nabla w(\tau,s)|^2 \mu(s)^{1-\sigma} dx ds d\tau   \notag\\
&\leq C\sup_{\tau \in\mathbb R} \|\nabla u(\tau)\|_2^2 \int_0^t \int_0^{\infty} \mu(s)^{1-\sigma} ds d\tau \notag\\
&\leq C(M_0,E(0)) t \left(\int_0^{\infty} \mu(s)^{1-\sigma}ds\right) \leq C(r,\sigma,M_0,E(0))t,
\end{align}
where we have used (\ref{cB-5}).

By (\ref{cB-1}) and (\ref{cB-2}), and the assumption $\mu'(s)+C\mu(s)^r \leq 0$, we obtain
\begin{align} \label{cB-6}
&\int_0^t \int_0^{\infty} \|\nabla w(\tau,s)\|_2^2 \mu(s) ds    d\tau \notag\\
&\leq C(r, \sigma,M_0,E(0)) t^{\frac{r-1}{\sigma+r-1}} \left(\int_0^t \int_0^{\infty}  \|\nabla w(\tau,s)\|_2^2 \left(-\mu'(s)\right) ds  d\tau \right)^{\frac{\sigma}{\sigma+r-1}} \notag\\
&\leq C(r, \sigma,M_0,E(0))  t  \textbf{D}(t)^{\frac{\sigma}{\sigma+r-1}},
\end{align}
for $t\geq 1$, where $1<r<2$, $0<\sigma<2-r$.

\vspace{0.1 in}

In the above calculations, we have finished the estimates of all the terms on the right-hand side of (\ref{w4}). Therefore, by applying estimates (\ref{w5p}), (\ref{w5}), (\ref{w15}), (\ref{w3-2}) and (\ref{w21}) into (\ref{w4}), we have 
\begin{align}    \label{w25}
&2 \int_0^t \mathscr E(\tau) d\tau
\leq  \left[   2\epsilon k(0)    + \tilde C(E(0))     \right]   \int_0^t \mathscr E(\tau) d\tau +C_{\epsilon}\int_0^t \int_0^{\infty} \|\nabla w(\tau,s)\|_2^2 \mu(s) ds d\tau   \notag\\
&\hspace{1 in} + C t  \textbf{D}(t)^{\frac{2}{m+1}}   +  C(\epsilon, E(0)) \textbf{D}(t) + C E(t).
\end{align}

By (\ref{tildeC}), we know that $\tilde C(E(0))<2$. So by choosing $\epsilon>0$ sufficiently small depending on $E(0)$, we deduce that
\begin{align}    \label{w25}
\int_0^t \mathscr E(\tau) d\tau \leq & C(E(0)) \left(\int_0^t \int_0^{\infty} \|\nabla w(\tau,s)\|_2^2 \mu(s) ds d\tau + t  \textbf{D}(t)^{\frac{2}{m+1}}   +   \textbf{D}(t) +  E(t) \right).
\end{align}
Notice that $\mathscr E(t) \geq E(t)$ and $E(t)$ is non-increasing for all $t\geq 0$. Thus
\begin{align}  \label{w26}
\int_0^t \mathscr E(\tau) d\tau \geq \int_0^t E(\tau) d\tau \geq t E(t).
\end{align}
Due to (\ref{w25}) and (\ref{w26}), and for  sufficiently large $t$, depending on $E(0)$, we obtain that,
\begin{align}    \label{w27}
t E(t)  \leq &C(E(0)) \left(\int_0^t \int_0^{\infty} \|\nabla w(\tau,s)\|_2^2 \mu(s) ds d\tau + t  \textbf{D}(t)^{\frac{2}{m+1}}   +  \textbf{D}(t) \right).
\end{align}

Next, we consider two different assumptions on the relaxation kernel $\mu(s)$. On the one hand, if $\mu'(s) + C\mu(s)\leq 0$, then by (\ref{w27}) and (\ref{w29}), it follows that
\begin{align}  \label{w30'}
tE(t) \leq C(E(0)) \left(t\textbf{D}(t)^{\frac{2}{m+1}} +\textbf{D}(t) \right),
\end{align}
for sufficiently large $t$.   By dividing both sides of (\ref{w30'}) by $t\geq 1$, one has
\begin{align}  \label{w30}
E(t) \leq C(E(0)) \left(\textbf{D}(t)^{\frac{2}{m+1}} +\textbf{D}(t) \right).
\end{align}

On the other hand, if $\mu'(s) + C \mu(s)^r \leq 0$, where $r\in (1,2)$, then by (\ref{w27}) and (\ref{cB-6}), one has
\begin{align}   \label{w31'}
tE(t)\leq C(E(0)) \left( C(r,\sigma,M_0,E(0)) t\textbf{D}(t)^{\frac{\sigma}{\sigma+r-1}}+      t\textbf{D}(t)^{\frac{2}{m+1}} +\textbf{D}(t) \right),
\end{align}
for $\sigma \in (0,2-r)$ and $t$ sufficiently large. By dividing both sides of (\ref{w31'}) by $t\geq 1$, we obtain
\begin{align}   \label{w31}
E(t)\leq C(r,\sigma,M_0,E(0)) \textbf{D}(t)^{\frac{\sigma}{\sigma+r-1}} + C(E(0))\left(\textbf{D}(t)^{\frac{2}{m+1}} +\textbf{D}(t) \right),
\end{align}
completing the proof for Proposition \ref{prop-pert}.
\end{proof}

\vspace{0.1 in}

\subsection{Derivation of the energy decay rates}  \label{sec-rates}
We employ the method introduced by Lasiecka and Tataru in \cite{LT} to derive the energy decay rates. From the energy identity (\ref{EI-1'}), one has $\textbf{D}(t)=E(0)-E(t)$. By Proposition \ref{prop-pert}, and if $\mu'(s)+C\mu(s)\leq 0$, then
\begin{align} \label{ab-23}
E(t)\leq \Phi(\textbf{D}(t)) = \Phi(E(0)-E(t))
\end{align}
for sufficiently large $t$, where $\Phi(s)=C(E(0))(s^{\frac{2}{m+1}}+s)$. Clearly, $\Phi$ is monotone increasing and vanishing at the origin. Thus, (\ref{ab-23}) implies that 
\begin{align}   \label{ab-23'}
(I+\Phi^{-1})E(t) \leq E(0),  
\end{align}
for $t$ sufficiently large.

If $\mu'(s)+C \mu(s)^r \leq 0$ where $r\in (1,2)$, then for any $\sigma \in (0,2-r)$, one has
\begin{align} \label{ab-24}
E(t)\leq \Psi(\textbf{D}(t)) = \Psi(E(0)-E(t)),
\end{align}
for sufficiently large $t$, where $\Psi(s)=C(r,\sigma,M_0,E(0)) s^{\frac{\sigma}{\sigma+r-1}}+C(E(0))\left(s^{\frac{2}{m+1}}+s\right)$. Then 
\begin{align}   \label{ab-24'}
(I+\Psi^{-1})E(t) \leq E(0),  
\end{align}
for sufficiently large $t$.

Let us first consider the case $\mu'(s)+C\mu(s)\leq 0$. By fixing a sufficiently large time $T$ depending on $E(0)$, we obtain from (\ref{ab-23'}) that $(I+\Phi^{-1})E(T) \leq E(0)$. Since $E(t)$ is monotone decreasing, we can reiterate the estimate $n+1$ times, for any $n=0,1,2,\cdots$, we obtain
\begin{align}   \label{ta1}
\left(I+\Phi^{-1}\right)E((n+1)T) \leq E(nT),     \text{\;\;for all\;\;}   n=0,1,2, \cdots .
\end{align}

By (\ref{ta1}), we shall show that $E(t)$ is bounded by the solution of an ordinary differential equation. The idea is from \cite{LT}.
Indeed, we claim that 
\begin{align}   \label{ta2}
E(nT)\leq S(n) \text{\;\;for all\;\;}   n=0,1,2, \cdots ,
\end{align}
where $S(t)$ is the solution of the initial value problem:
\begin{align}  \label{ta3}
S'(t) + (I+\Phi)^{-1}S(t)=0,      \;\;\; S(0)=E(0).
\end{align}
Claim (\ref{ta2}) can be proved by induction. Since $S(0)=E(0)$, (\ref{ta2}) is valid for $n=0$. Assume $E(nT)\leq S(n)$ for a given $n\geq 0$, we show $E((n+1)T)\leq S(n+1)$.
Clearly, (\ref{ta1}) implies
\begin{align}   \label{ta4}
E((n+1)T) \leq (I+\Phi^{-1})^{-1}E(nT).
\end{align}
Also, since $(I+\Phi)^{-1}$ is monotone increasing on $[0,\infty)$ and vanishing at the origin, we obtain from (\ref{ta3}) that $S(t)$ 
monotonically decreases to zero as $t \rightarrow \infty$. Therefore, by integrating (\ref{ta3}) from $n$ to $n+1$, one has
\begin{align} \label{ta5}
S(n+1) = S(n)-\int_n^{n+1} (I+\Phi)^{-1}S(\tau) d\tau  \geq S(n) - (I+\Phi)^{-1}S(n).
\end{align}
Notice that
\begin{align}    \label{ta6}
(I+&\Phi^{-1})^{-1} =   \Phi \circ  \Phi^{-1}  \circ   (I+\Phi^{-1})^{-1}  =  \Phi  \circ  (\Phi+I)^{-1}      \notag\\
&= (\Phi+ I) \circ  (\Phi+I)^{-1}  -  (\Phi+I)^{-1}  = I -    (\Phi+I)^{-1} .
\end{align}
Applying (\ref{ta6}) to (\ref{ta5}) implies
\begin{align}     \label{ta7}
S(n+1) \geq (I+\Phi^{-1})^{-1} S(n) \geq (I+\Phi^{-1})^{-1}E(nT),
\end{align}
where the last inequality is due to the induction hypothesis as well as the fact that $(I+\Phi^{-1})^{-1}$ is monotone increasing. Combining of (\ref{ta4}) and (\ref{ta7}) yields $E((n+1)T) \leq S(n+1)$. This completes the proof of  claim (\ref{ta2}).

Next we use (\ref{ta2}) and (\ref{ta3}) to calculate the uniform decay rate of the total energy $E(t)$ as $t\rightarrow \infty$. For any $t>T$, there exists $n\in \mathbb N$ such that $t=nT+b$ where $b\in [0,T)$. Thus, $n=\frac{t-b}{T}>\frac{t}{T}-1$. Since $E(t)$ and $S(t)$ are monotone decreasing in time, by applying (\ref{ta2}) one has
\begin{align}   \label{ta8}
E(t) = E(nT+b) \leq E(nT) \leq S(n)  \leq S\left(\frac{t}{T}-1\right), \text{\;\;for any\;\;}  t>T.
\end{align}

If $m=1$, then $\Phi(s)=C(E(0))s$ is linear, then the differential equation (\ref{ta3}) is also linear. Thus $S(t)$ decays to zero exponentially fast. It follows from (\ref{ta8}) that $E(t)$ also decays to zero exponentially fast. That is,
\begin{align*}
E(t) \leq C(E(0))  e^{-\alpha t},  \text{\;\;\;for\;\;} t\geq 0,
\end{align*}
where $\alpha$ depends on $E(0)$.

If $m>1$, then $\Phi(s)=C(E(0))(s^{\frac{2}{m+1}}+s) \leq C(E(0)) s^{\frac{2}{m+1}}$ if $0\leq s\leq 1$, that is to say, for sufficiently small value of $s$, the function $\Phi(s)$ has a polynomial growth rate $\frac{2}{m+1}$. Recall that $S(t)\rightarrow 0$ as $t\rightarrow \infty$. Consequently, by ODE (\ref{ta3}), we obtain that, for $t$ large enough, $S'(t)+C(E(0))S(t)^{\frac{m+1}{2}} \leq 0$ with $S(0)=E(0)$. 
Thus, $S(t) \leq C(E(0)) (1+t)^{-\frac{2}{m-1}}$. It follows from (\ref{ta8}) that the energy $E(t)$ also decays polynomially in time:
\begin{align*} 
E(t)\leq C(E(0)) (1+t)^{-\frac{2}{m-1}}, \text{\;\;\;for\;\;} t\geq 0.
\end{align*}

Notice that  inequalities (\ref{ab-24'}) and (\ref{ab-23'}) have the same form. Therefore, we can employ the exact same strategy as above, to study the case that $\mu'(s)+C\mu(s)^r \leq 0$ for all $s\geq 0$, for some $r\in (1,2)$, and to find the uniform energy decay rate. 
In fact, if $m=1$, then 
$\Psi(s)=C(r,\sigma,M_0,E(0)) s^{\frac{\sigma}{\sigma+r-1}}+C(E(0))s$. Then, we deduce that the energy $E(t)$ decays polynomially in time:
\begin{align} \label{polymu}
E(t)\leq C(r,\sigma,M_0,E(0)) (1+t)^{-\frac{\sigma}{r-1}},  \text{\;\;\;for\;\;} t\geq 0,
\end{align}
for any $\sigma\in (0,2-r)$. However, if $m>1$, then $\Psi(s)=C(r,\sigma,M_0,E(0)) s^{\frac{\sigma}{\sigma+r-1}}+C(E(0))(s^{\frac{2}{m+1}}+s)$ for $r\in (1,2)$ and 
$\sigma\in (0,2-r)$. In this case, we derive that the energy $E(t)$ decays polynomially fast:
\begin{align}    \label{polymu2}
E(t)\leq C(r,\sigma,M_0, E(0)) (1+t)^{-\max\left\{\frac{\sigma}{r-1},\frac{2}{m-1}\right\}},  \text{\;\;\;for\;\;} t\geq 0.
\end{align}

In the next section, by imposing an additional assumption on the history value $u_0$, we shall improve the energy decay rates in (\ref{polymu}) and (\ref{polymu2}) 
when the relaxation kernel $\mu(s)$ satisfying $\mu'(s)+C\mu(s)^r \leq 0$ for all $s\geq 0$, for some $r\in (1,2)$.

\vspace{0.1 in}

\subsection{Optimal polynomial decay rates}
In this subsection, we assume there exists $T_0>0$ such that the history value $u_0$ is supported on $[-T_0,0]$, that is $u_0(t)=0$ for $t\leq -T_0$. When the relaxation kernel $\mu(s)$ satisfies $0<\mu(s)\leq C(1+s)^{-\frac{1}{r-1}}$ for all $s>0$, for some $r\in (1,2)$, we shall demonstrate an algorithm for improving a non-optimal polynomial decay rate of the energy to the optimal one. We say the energy decay rate is \emph{optimal} if it is consistent with the decay rate of the relaxation kernel $\mu(s)$, namely $E(t) \leq C (1+t)^{-\frac{1}{r-1}}$ for all $t\geq 0$. 

To begin with the algorithm, we need an existing decay rate (not optimal): 
\begin{align}     \label{old-rate}
E(t) \leq C   (1+t)^{-\frac{\sigma_1}{r-1}},  \;\;\text{for some}  \; 0<\sigma_1<1 \; \text{with}\; \sigma_1 \not=r-1, \;\;\text{for all}\; t\geq 0.
\end{align}
Notice that the decay rate in (\ref{old-rate}) is guaranteed by (\ref{polymu}) and (\ref{polymu2}) in the previous subsection. Since $\frac{1}{2}\|\nabla u(t)\|_2^2 \leq \mathscr E(t) \leq \frac{p+1}{p-1} E(t)$, we obtain from (\ref{old-rate}) that
\begin{align}  \label{old-rate1}
\|\nabla u(t)\|_2^2 \leq  C(1+t)^{-\frac{\sigma_1}{r-1}}, \;\;\text{for}\;\; t\geq 0.
\end{align}
The essence of the algorithm is that one takes advantage of the polynomial decay of $\|\nabla u(t)\|_2^2$ shown in (\ref{old-rate1}) when performing the estimate in order to obtain an improved energy decay rate $E(t)\leq C (1+t)^{-\frac{\sigma_2}{r-1}}$, with $\sigma_1<\sigma_2\leq 1$. 
We aim to show that by iterating the following procedure finitely many times, one will eventually achieve the optimal polynomial decay rate $E(t) \leq C(1+t)^{-\frac{1}{r-1}}$, for $t\geq 0$. The idea of the algorithm is from \cite{Cav5,LMM,LW}. In particular, we would like to mention a recent work \cite{LW} by Lasiecka and Wang, in which 
optimal energy decay rates are obtained for semilinear abstract second-order equations influenced by a finite memory term with a general relaxation kernel. However, we remark that in \cite{Cav5,LMM,LW} and many other relevant research in the literature, the initial value $u_0$ is defined at an instant $t=0$. Here we have a different setting: the history value $u_0$ is supported on a time interval $[-T_0,0]$ for some $T_0>0$.

Notice that
\begin{align}  \label{op-2}
&|\nabla w(\tau,s)|^2 - \nabla w(\tau,s) \cdot \nabla u(\tau)  \notag\\
&=|\nabla u(\tau) - \nabla u(\tau-s)|^2 - (\nabla u(\tau)- \nabla u(\tau-s)) \cdot \nabla u(\tau) \notag\\
&= |\nabla u(\tau)|^2 - 2 \nabla u(\tau) \cdot \nabla u(\tau-s) + |\nabla u(\tau-s)|^2 -  |\nabla u(\tau)|^2 + \nabla u(\tau-s) \cdot u(\tau)  \notag\\
&= - \nabla u(\tau) \cdot \nabla u(\tau-s)  +   |\nabla u(\tau-s)|^2.
\end{align}
By using (\ref{op-2}) and the assumption that $u_0$ is supported on $[-T_0,0]$, we estimate the following two terms from (\ref{w4}):
\begin{align}    \label{op-0}
&\int_0^t \int_0^{\infty} \|\nabla w(\tau,s)\|_2^2 \mu(s) ds d\tau   - \int_0^t \int_0^{\infty} \int_{\Omega} \nabla w(\tau,s) \cdot \nabla u(\tau) dx \mu(s) ds d\tau  \notag\\
&=\int_0^t \int_0^{\infty} \int_{\Omega}   \left(-\nabla u(\tau) \cdot \nabla u(\tau-s)  +   |\nabla u(\tau-s)|^2\right) dx \mu(s) ds d\tau  \notag\\
&=\int_0^t  \int_0^{\tau+T_0} \int_{\Omega}    \left(-\nabla u(\tau) \cdot \nabla u(\tau-s)  +   |\nabla u(\tau-s)|^2\right) dx \mu(s) ds d\tau  \notag\\
&=\int_0^t \int_0^{\tau+T_0} \|\nabla w(\tau,s)\|_2^2 \mu(s) ds d\tau   - \int_0^t \int_0^{\tau+T_0} \int_{\Omega} \nabla w(\tau,s) \cdot \nabla u(\tau) dx \mu(s) ds d\tau.
\end{align}
Note $\int_0^{\tau+T_0} \mu(s) ds = - \int_0^{\tau+T_0} k'(s) ds = k(0) - k(\tau+T_0) < k(0) -1$ since $k(s)$ is monotone decreasing with $k(\infty)=0$. Then
thanks to Cauchy--Schwarz and Young's inequalities, one has
\begin{align}    \label{op-3}
&\int_0^t \int_0^{\tau+T_0} \int_{\Omega} \nabla w(\tau,s) \cdot \nabla u(\tau) dx \mu(s) ds d\tau \notag\\
&\leq  \epsilon (k(0)-1) \int_0^t \mathscr E(\tau) d\tau + C_{\epsilon}   \int_0^t \int_0^{\tau+T_0} \|\nabla w(\tau,s)\|_2^2 \mu(s) ds d\tau,
\end{align}
where we have used $\|\nabla u(t)\|_2^2 \leq 2 \mathscr E(t)$. By applying (\ref{op-3}) to (\ref{op-0}), we obtain
\begin{align}   \label{op-3'}
&\int_0^t \int_0^{\infty} \|\nabla w(\tau,s)\|_2^2 \mu(s) ds d\tau   - \int_0^t \int_0^{\infty} \int_{\Omega} \nabla w(\tau,s) \cdot \nabla u(\tau) dx \mu(s) ds d\tau \notag\\
&\leq \epsilon (k(0)-1) \int_0^t \mathscr E(\tau) d\tau + C_{\epsilon}   \int_0^t \int_0^{\tau+T_0} \|\nabla w(\tau,s)\|_2^2 \mu(s) ds d\tau.
\end{align}

Now we proceed with our estimate for two different cases regarding the value of $\sigma_1$.

\vspace{0.1 in}

\textbf{\emph{Case I}}:  $r-1<\sigma_1<1$. 

\vspace{0.05 in}

By H\"older's inequality, we have
\begin{align}  \label{op-4}
&\int_0^t \int_0^{\tau+T_0}  \int_{\Omega} |\nabla w(\tau,s)|^2  \mu(s) dx  ds d\tau   \notag\\
&\leq    \left(\int_0^t \int_0^{\tau+T_0}  \int_{\Omega} |\nabla w(\tau,s)|^2  dx  ds d\tau\right)^{\frac{r-1}{r}}
\left(\int_0^t \int_0^{\tau+T_0}  \int_{\Omega} |\nabla w(\tau,s)|^2  \mu(s)^r dx  ds d\tau  \right)^{\frac{1}{r}}.
\end{align}

We estimate
\begin{align}    \label{op-1}
&\int_0^{\tau+T_0}  \|\nabla w(\tau,s)\|_2^2 ds  
= \int_0^{\tau+T_0}  \|\nabla u(\tau)-\nabla u(\tau-s)\|_2^2 ds    \notag\\
&\leq 2\int_0^{\tau+T_0} \|\nabla u(\tau)\|_2^2 ds  + 2\int_0^{\tau+T_0}  \|\nabla u(\tau-s)\|_2^2 ds  \notag\\
&= 2(\tau+T_0) \|\nabla u(\tau)\|_2^2  +   2\int_0^{\tau} \|\nabla u(\tau-s)\|_2^2 ds + 2\int_{\tau}^{\tau+T_0} \|\nabla u(\tau-s)\|_2^2 ds   \notag\\
&=2(\tau+T_0) \|\nabla u(\tau)\|_2^2 +   2\int_0^{\tau} \|\nabla u(s)\|_2^2 ds + 2\int_{-T_0}^0 \|\nabla u_0(s)\|_2^2 ds,
\end{align}
for $\tau\geq 0$. By using the decay estimate $\|\nabla u(\tau)\|_2^2 \leq  C(1+\tau)^{-\frac{\sigma_1}{r-1}}$ from (\ref{old-rate1}) and the assumption $\sigma_1>r-1$, we obtain from (\ref{op-1}) that
\begin{align*}
\int_0^{\tau+T_0}  \|\nabla w(\tau,s)\|_2^2 ds  \leq  C(r,\sigma_1)\left(1 +T_0(1+\tau)^{-\frac{\sigma_1}{r-1}}  \right) + 2\int_{-\infty}^0 \|\nabla u_0(s)\|_2^2 ds, \;\; \text{for}\;\;\tau\geq 0.
\end{align*}
Integrating the above inequality with respect to $\tau$ over $[0,t]$ yields
\begin{align}   \label{op-6}
\int_0^t \int_0^{\tau+T_0}      \|\nabla w(\tau,s)\|_2^2 ds d\tau    \leq   C(r,\sigma_1)(T_0+t)+  2t\int_{-\infty}^0 \|\nabla u_0(s)\|_2^2 ds,
\end{align}
due to the assumption $\sigma_1>r-1$.

Substituting (\ref{op-6}) into (\ref{op-4}) and using the assumption $\mu'(s)+ C\mu(s)^r \leq 0$, we find
\begin{align}    \label{op-8}
&\int_0^t \int_0^{\tau+T_0}  \|\nabla w(\tau,s)\|^2  \mu(s)  ds d\tau  \notag\\
&\leq \left(C(r,\sigma_1)(T_0+t) +   2t\int_{-\infty}^0 \|\nabla u_0(s)\|_2^2 ds\right)^{\frac{r-1}{r}}   \left(\int_0^t \int_0^{T_0}  \|\nabla w(\tau,s)\|^2_2  (-\mu'(s))   ds d\tau  \right)^{\frac{1}{r}} \notag\\
&\leq  \left(C(r,\sigma_1)(T_0+t) +   2t\int_{-\infty}^0 \|\nabla u_0(s)\|_2^2 ds\right) \mathbf D(t)^{\frac{1}{r}},   \;\;\;  \text{for} \;\; t\geq 0,
\end{align}
by letting $C(r,\sigma_1)T_0 \geq 1$.

Applying (\ref{op-8}) to the right-hand side of (\ref{op-3'}) yields
\begin{align}  \label{op-8'}
&\int_0^t \int_0^{\infty} \|\nabla w(\tau,s)\|_2^2 \mu(s) ds d\tau   - \int_0^t \int_0^{\infty} \int_{\Omega} \nabla w(\tau,s) \cdot \nabla u(\tau) dx \mu(s) ds d\tau \notag\\
& \leq  \epsilon (k(0)-1) \int_0^t \mathscr E(\tau) d\tau +   C_{\epsilon} \left(C(r,\sigma_1)(T_0+t)+   t\int_{-\infty}^0 \|\nabla u_0(s)\|_2^2 ds\right) \mathbf D(t)^{\frac{1}{r}}.
\end{align}

Using (\ref{op-8'}) along with estimates (\ref{w5}), (\ref{w15}), (\ref{tildeC}), (\ref{w21}) and (\ref{w26}) derived in subsection \ref{sec-stab}, and by choosing $\epsilon>0$ sufficiently small, we deduce from (\ref{w4}) that
\begin{align*}
t E(t)  \leq C(E(0))    \left[\left(C(r,\sigma_1)(T_0+t)+   t\int_{-\infty}^0 \|\nabla u_0(s)\|_2^2 ds\right) \mathbf D(t)^{\frac{1}{r}}  
+t \mathbf D(t)^{\frac{2}{m+1}}  + \mathbf D(t)  \right],
\end{align*}
for a sufficiently large $t$ depending on $E(0)$. Dividing both sides of the above inequality by $t\geq T_0 \geq 1$, it follows that
\begin{align} \label{op-9}
E(t) \leq  C(E(0))    \left[\left(C(r,\sigma_1) +  \int_{-\infty}^0 \|\nabla u_0(s)\|_2^2 ds\right) \mathbf D(t)^{\frac{1}{r}} +\mathbf D(t)^{\frac{2}{m+1}}  + \mathbf D(t)  \right].
\end{align}

Let us first consider $m=1$. Then employing the same procedure as in the subsection \ref{sec-rates}, we obtain from (\ref{op-9}) that, 
for a sufficiently large $T\geq T_0$, one has 
\begin{align}  \label{opp}
E(t) \leq S\left(\frac{t}{T}-1\right), \;\;\text{for}  \;\; t>T \geq T_0,
\end{align}
where $S(t)$ satisfies the differential equation
\begin{align*}
S'(t) + C S(t)^{r} \leq 0 \;\;\text{with}\;\; S(0)=E(0),\;\; \text{if} \;\; m=1,
\end{align*}
where $C$ depends on $r$, $\sigma_1$, $E(0)$, and $\int_{-\infty}^0 \|\nabla u_0(s)\|_2^2 ds.$ Hence, $S(t)\leq C (1+t)^{-\frac{1}{r-1}}$ for $t\geq 0$, and along with (\ref{opp}),
we obtain the optimal polynomial decay rate
\begin{align*}
E(t) \leq \tilde C(1+t)^{-\frac{1}{r-1}}, \;\;\text{for all} \;\; t \geq 0,
\end{align*}
where $\tilde C$ depend on $r$, $\sigma_1$, $E(0)$, $\int_{-\infty}^0 \|\nabla u_0(s)\|_2^2 ds$ and $T_0$, when $m=1$.

Analogously, if $m>1$, using estimate (\ref{op-9}), we also obtain the optimal polynomial decay rate
\begin{align*}
E(t) \leq \tilde C(1+t)^{-\max\left\{\frac{1}{r-1}, \;\frac{2}{m-1} \right\}}, \;\;\text{for} \;\; t \geq 0,
\end{align*}
where $\tilde C$ depend on $r$, $\sigma_1$, $E(0)$, $\int_{-\infty}^0 \|\nabla u_0(s)\|_2^2 ds$ and $T_0$.

\vspace{0.1 in}
\textbf{\emph{Case II}}: $0<\sigma_1<r-1$.

\vspace{0.05 in}

In this case, we select $\sigma_2>\sigma_1$ such that 
\begin{align} \label{sigma2}
 2-r<\sigma_2<2-r+\sigma_1<1.
 \end{align}
We aim to derive an improved polynomial decay rate $E(t)\leq C (1+t)^{-\frac{\sigma_2}{r-1}}$, for $t\geq 0$.
Thanks to H\"older's inequality, we have
\begin{align}  \label{op-11}
&\int_0^t \int_0^{\tau+T_0}  \int_{\Omega} |\nabla w(\tau,s)|^2  \mu(s) dx  ds d\tau  \notag\\
&\leq \left(\int_0^t \int_0^{\tau+T_0}  \int_{\Omega} |\nabla w(\tau,s)|^2 \mu(s)^{1-\sigma_2}  dx  ds d\tau\right)^{\frac{r-1}{\sigma_2+r-1}}    \notag\\
& \hspace{1 in} \left(\int_0^t \int_0^{\tau+T_0}  \int_{\Omega} |\nabla w(\tau,s)|^2  \mu(s)^r dx  ds d\tau  \right)^{\frac{\sigma_2}{\sigma_2+ r -1}} \notag\\
&\leq    \left(\int_0^t \int_0^{\tau+T_0}  \int_{\Omega} |\nabla w(\tau,s)|^2 \mu(s)^{1-\sigma_2}  dx  ds d\tau\right)^{\frac{r-1}{\sigma_2+r-1}}  
\mathbf D(t)^{\frac{\sigma_2}{\sigma_2+ r -1}},
\end{align}
where we have used $\mu'(s)+C\mu(s)^r \leq 0$ for $s\geq 0$.

Using $\mu(s)\leq C(1+s)^{-\frac{1}{r-1}}$ for all $s\geq 0$, we infer
\begin{align}   \label{op-12}
&\int_0^{\tau+T_0}  \|\nabla w(\tau,s)\|_2^2 \mu(s)^{1-\sigma_2} ds  
= \int_0^{\tau+T_0}  \|\nabla u(\tau)-\nabla u(\tau-s)\|_2^2 \mu(s)^{1-\sigma_2} ds    \notag\\
&\leq C\int_0^{\tau+T_0} \|\nabla u(\tau)\|_2^2  (1+s)^{-\frac{1-\sigma_2}{r-1}}   ds  + C\int_0^{\tau+T_0}  \|\nabla u(\tau-s)\|_2^2 (1+s)^{-\frac{1-\sigma_2}{r-1}}  ds.
\end{align}

Since $2-r<\sigma_2<1$ and using $\|\nabla u(\tau)\|_2^2 \leq  C(1+\tau)^{-\frac{\sigma_1}{r-1}}$ for $\tau\geq 0$ from (\ref{old-rate1}), one has
\begin{align}   \label{op-13}
&\int_0^{\tau+T_0} \|\nabla u(\tau)\|_2^2  (1+s)^{-\frac{1-\sigma_2}{r-1}}ds  
=    \|\nabla u(\tau)\|_2^2   \int_0^{\tau+T_0}  (1+s)^{-\frac{1-\sigma_2}{r-1}}ds  \notag\\
&=C(r,\sigma_2) \left((1+\tau+T_0)^{\frac{r+\sigma_2-2}{r-1}}-1\right) \|\nabla u(\tau)\|_2^2     \notag\\
&\leq C(r,\sigma_2)   (1+\tau+T_0)^{\frac{r+\sigma_2-2}{r-1}}    (1+\tau)^{-\frac{\sigma_1}{r-1}}  \notag\\
&\leq C(r, \sigma_2) \left((1+\tau)^{\frac{r-\sigma_1 +\sigma_2 -2}{r-1}}   +    T_0^{\frac{r+\sigma_2-2}{r-1}} (1+\tau)^{-\frac{\sigma_1}{r-1}} \right),
\;\;\;\text{for}\;\; \tau\geq 0.
\end{align}

Next we estimate the second term on the right-hand side of (\ref{op-12}). Also using $\|\nabla u(\tau)\|_2^2 \leq  C(1+\tau)^{-\frac{\sigma_1}{r-1}}$ for all $\tau\geq 0$, we have
\begin{align}   \label{op-14}
&\int_0^{\tau+T_0}  \|\nabla u(\tau-s)\|_2^2 (1+s)^{-\frac{1-\sigma_2}{r-1}}  ds   \notag\\
&=  \int_0^{\tau}    \|\nabla u(\tau-s)\|_2^2 (1+s)^{-\frac{1-\sigma_2}{r-1}}  ds 
+ \int_{\tau}^{\tau+T_0}    \|\nabla u(\tau-s)\|_2^2 (1+s)^{-\frac{1-\sigma_2}{r-1}}  ds \notag\\
&\leq  C\int_0^{\tau}     (1+\tau-s)^{-\frac{\sigma_1}{r-1}}  (1+s)^{-\frac{1-\sigma_2}{r-1}}   ds 
+ \int_{-T_0}^{0} \|\nabla u_0(s)\| ds,    \;\;\;\text{for}\;\; \tau\geq 0.
\end{align}

Recall $0<\sigma_1<r-1$ and $2-r<\sigma_2<1$. Then we calculate
\begin{align}    \label{op-14'}
&\int_0^{\tau}     (1+\tau-s)^{-\frac{\sigma_1}{r-1}}  (1+s)^{-\frac{1-\sigma_2}{r-1}}   ds   \notag\\
&\leq   \int_0^{\tau/2}    \left(1+\frac{\tau}{2}\right)^{-\frac{\sigma_1}{r-1}}  (1+s)^{-\frac{1-\sigma_2}{r-1}}   ds
+ \int_{\tau/2}^{\tau}   (1+\tau-s)^{-\frac{\sigma_1}{r-1}}  \left(1+\frac{\tau}{2}\right)^{-\frac{1-\sigma_2}{r-1}}   ds  \notag\\
&\leq    C(r,\sigma_2) \left(1+\frac{\tau}{2}\right)^{-\frac{\sigma_1}{r-1}}    \left(1+\frac{\tau}{2}\right)^{\frac{r+\sigma_2-2}{r-1}} 
+ C(r,\sigma_1)   \left(1+\frac{\tau}{2}\right)^{-\frac{1-\sigma_2}{r-1}}  \left(1+\frac{\tau}{2}\right)^{\frac{r-1-\sigma_1}{r-1}} \notag\\
&\leq C(r,\sigma_1,\sigma_2)   (1+\tau)^{\frac{r-\sigma_1+\sigma_2-2}{r-1}},      \;\;\;\text{for}\;\; \tau\geq 0.
\end{align}

Substituting (\ref{op-14'}) into (\ref{op-14}), we find
\begin{align}    \label{op-14''}
&\int_0^{\tau+T_0}  \|\nabla u(\tau-s)\|_2^2 (1+s)^{-\frac{1-\sigma_2}{r-1}}  ds     \notag\\
&\leq  C(r,\sigma_1,\sigma_2) (1+\tau)^{\frac{r-\sigma_1+\sigma_2-2}{r-1}} + \int_{-\infty}^{0} \|\nabla u_0(s)\| ds,
 \;\;\;\text{for}\;\; \tau\geq 0.
\end{align}  

Now, applying (\ref{op-13}) and (\ref{op-14''}) into (\ref{op-12}) yields
\begin{align}  \label{op-15}
&\int_0^{\tau+T_0}  \|\nabla w(\tau,s)\|_2^2 \mu(s)^{1-\sigma_2} ds 
\leq   C(r,\sigma_1,\sigma_2)    (1+\tau)^{\frac{r-\sigma_1+\sigma_2-2}{r-1}}  \notag\\
&+  C(r,\sigma_2) T_0^{\frac{r+\sigma_2-2}{r-1}} (1+\tau)^{-\frac{\sigma_1}{r-1}}   +      \int_{-\infty}^{0} \|\nabla u_0(s)\| ds,  \;\;\;\text{for}\;\; \tau\geq 0.
\end{align}

Recall $2-r<\sigma_2<2-r+\sigma_1$ and $0<\sigma_1<r-1$, then $0<\frac{r-\sigma_1+\sigma_2-2}{r-1}+1<1$ and $0<1-\frac{\sigma_1}{r-1}<1$. 
Also, $0<\frac{r+\sigma_2-2}{r-1}<1$. Therefore, integrating (\ref{op-15}) with respect to $\tau$ over the interval $[0,t]$ yields
\begin{align}   \label{op-17}
&\int_0^t \int_0^{\tau+T_0}  \|\nabla w(\tau,s)\|_2^2 \mu(s)^{1-\sigma_2} ds  d\tau  \notag\\
&\leq     C(r,\sigma_1,\sigma_2)  (1+t)^{\frac{r-\sigma_1+\sigma_2-2}{r-1}+1}   + C(r,\sigma_2) T_0^{\frac{r+\sigma_2-2}{r-1}}(1+t)^{1-\frac{\sigma_1}{r-1}}  
+ t \int_{-\infty}^0 \|\nabla u_0(s)\|_2^2 ds  \notag\\
&\leq     C(r,\sigma_1,\sigma_2) (1+t)  +  C(r,\sigma_2)  T_0 (1+t) +     t \int_{-\infty}^0 \|\nabla u_0(s)\|_2^2 ds,    \;\;\;\text{for}\;\;t\geq 0,
\end{align}  
by letting $T_0\geq 1$.

Substituting (\ref{op-17}) into (\ref{op-11}), we have
\begin{align}  \label{op-18}
&\int_0^t \int_0^{\tau+T_0}  \int_{\Omega} |\nabla w(\tau,s)|^2  \mu(s) dx  ds d\tau   \notag\\
&\leq   \left[C(r,\sigma_1,\sigma_2) (1+t)  +  C(r,\sigma_2)  T_0 (1+t) +     t \int_{-\infty}^0 \|\nabla u_0(s)\|_2^2 ds \right]
\mathbf D(t)^{\frac{\sigma_2}{\sigma_2+ r -1}},  
\end{align}
for $t\geq 0$, where $C(r,\sigma_2) T_0 \geq 1$.

The remaining argument is similar to Case I and we omit the detail. Eventually, we obtain
\begin{align}  
&E(t) \leq C(1+t)^{-\frac{\sigma_2}{r-1}},   \;\;  \text{if}  \;\; m=1;   \label{op-19}\\
&E(t) \leq C(1+t)^{-\max\left\{\frac{\sigma_2}{r-1}, \;\frac{2}{m-1} \right\}},   \;\;    \text{if}  \;\; m>1,  \label{op-20}
\end{align}
for all $t\geq 0$, where $C$ depends on $r$, $E(0)$, $\sigma_1$, $\sigma_2$, $\int_{-\infty}^0 \|\nabla u_0(s)\|_2^2 ds$ and $T_0$.

Recall $2-r<\sigma_2<2-r+\sigma_1<1$ in Case II, so the polynomial decay rates in (\ref{op-19}) and (\ref{op-20}) are not optimal. 
Thus, we need to reiterate the above procedure for finitely many times until the optimal decay rate is achieved. More precisely, if $\sigma_n>r-1$ for some $n\in \mathbb N$, then by Case I, the optimal decay rate can be obtained. However, if $\sigma_n<r-1$ for some $n\in \mathbb N$, then by Case II, we can find $\sigma_{n+1}>\sigma_n$ such that  $\sigma_{n+1}<2-r+\sigma_n$ 
and $E(t)\leq C (1+t)^{-\frac{\sigma_{n+1}}{r-1}}$, for $t\geq 0$.
In particular, we may choose
\begin{align*}
\sigma_{n+1} = \frac{2-r}{2}+\sigma_n =  \frac{(2-r)n}{2}+\sigma_1.
\end{align*}
After finitely many times of iteration, one obtains the optimal polynomial decay rate. This completes the proof of Theorem \ref{thm-decay}.
\vspace{0.1 in}

\begin{remark}
Notice that, if the relaxation kernel $\mu(s)$ decreases polynomially fast, we have obtained the optimal polynomial decay rate for the energy by assuming the history value is supported on an time interval $[-T_0,0]$ for some $T_0>0$. Although $T_0$ can be arbitrarily large, the constant $C$ in the energy decay estimate (see, for instance, (\ref{op-19}) and (\ref{op-20})) depends on $T_0$. So it is not obvious how to improve the polynomial decay rates in (\ref{polymu}) and (\ref{polymu2}) for the infinite memory case.   
\end{remark}

\noindent {\bf Acknowledgment.} Y. Guo would like to acknowledge support by the Simons collaboration grant.   
The authors are grateful to the anonymous reviewers for their valuable comments and suggestions to improve the quality of the paper.

\bibliographystyle{amsplain}

\end{document}